\documentclass{article}
\usepackage{graphicx} 
\usepackage{tikz}
\usepackage{circuitikz}
\usepackage{mathtools, stmaryrd}
\usepackage{amsfonts, amsmath, amssymb, amsthm}
\usepackage{color}

\title{Tasty Chocolate Games}
\author{Fran\c{c}ois Carret\thanks{ENS de Lyon\\ francois.carret@ens-lyon.fr} \and Koki Suetsugu\thanks{Toyo University\\ suetsugu.koki@gmail.com}}
\date{}

\newtheorem{theorem}{Theorem}
\newtheorem{lemma}{Lemma}
\newtheorem{definition}{Definition}

\newtheorem{remark}{Remark}

\newcommand{\MCHN}{\operatorname{MCHN}}
\newcommand{\CHN}{\operatorname{CHN}}
\newcommand{\CHCG}{\operatorname{CHCG}}
\newcommand{\DCG}{\operatorname{DCG}}
\newcommand{\CB}{\operatorname{ICG}}
\newcommand{\SCG}{\operatorname{RCG}}
\newcommand{\MRCG}{\operatorname{NIM}}

\begin{document}

\maketitle

\begin{abstract}

In this paper, we investigate Chocolate Games.
In Chocolate Games, each player cuts a chocolate bar into two chocolate bars and eats one of them so that they do not eat an indicated bitter block. 
 The player who eats the bitter part loses the game. This game can be considered as a generalization of Nim, and previous studies consider the condition that for what kind of the shape of chocolate bar, the Sprague--Grundy value of the position can be calculated by Nim--sum (XOR) of the width and height of the chocolate bar. Higher dimensional cases were also studied. 
 In this paper, we show that the ``Tasty condition,'' which was introduced in previous work for increasing staircase chocolate bars, can also be used for various other shapes of chocolate bars as a sufficient condition, or in some cases a necessary and sufficient condition, for the Sprague--Grundy value to be given by the Nim-sum of the sizes of the chocolate bar in each dimension.

\end{abstract}

\section{Introduction}


\subsection{Chocolate Games}
In this paper, we consider Chocolate Games.
In Chocolate Games, a chocolate bar is given as a position of the game. Some blocks of the chocolate bar is bitter. Each player, in their turn, cuts the chocolate bar into two bars along its groove and eats one of them. The player who eats the bitter part loses the game.

Chocolate Games are investigated in some previous studies like \cite{ MIF16, MNN20, MN21}.
One of the reasons why Chocolate Games are important is that Chocolate Games are generalizations of some well-known games. 
For example, some Chocolate Games can be considered as a position in Nim, which is one of the most well-known combinatorial games, solved in \cite{B01}.
In Nim, some piles of tokens are given as a position. Each player removes arbitrary many tokes from a pile. The player who moves last is the winner.
A position in Chocolate Game in which the chocolate bar is rectangular in shape and has a bitter block in one corner, like shown in Fig. \ref{fig:rcg57}, can be thought of as a two-pile Nim game since reducing the height or width corresponds to removing tokens from one of the piles. 
We assign coordinates to the blocks of the chocolate bar, with the bitter block at $(0,0)$. We call $(x+1) \times (y+1)$ rectangular chocolate bar $\SCG(x,y)$; a precise definition will be given later. Note that a single bitter block is $\SCG(0, 0)$.
Figures \ref{fig:hcutrcg} and \ref{fig:vcutrcg} show moves from $\SCG(7,5)$ by horizontal and vertical cuts.
\begin{figure}[!ht]
\centering
\resizebox{1\textwidth/4}{!}{%
\begin{circuitikz}
\tikzstyle{every node}=[font=\fontsize{18.2pt}{23.7pt}\selectfont]
\draw [color=white]  (-2,-1) rectangle (8,7);
\draw  (0,0) rectangle (1,1);
\draw  (1,0) rectangle (2,1);
\draw  (2,0) rectangle (3,1);
\draw  (3,0) rectangle (4,1);
\draw  (4,0) rectangle (5,1);
\draw  (5,0) rectangle (6,1);
\draw  (6,0) rectangle (7,1);
\draw  (-1,1) rectangle (0,2);
\draw  (0,1) rectangle (1,2);
\draw  (1,1) rectangle (2,2);
\draw  (2,1) rectangle (3,2);
\draw  (3,1) rectangle (4,2);
\draw  (4,1) rectangle (5,2);
\draw  (5,1) rectangle (6,2);
\draw  (6,1) rectangle (7,2);
\draw  (-1,2) rectangle (0,3);
\draw  (0,2) rectangle (1,3);
\draw  (1,2) rectangle (2,3);
\draw  (2,2) rectangle (3,3);
\draw  (3,2) rectangle (4,3);
\draw  (4,2) rectangle (5,3);
\draw  (5,2) rectangle (6,3);
\draw  (6,2) rectangle (7,3);
\draw  (-1,3) rectangle (0,4);
\draw  (0,3) rectangle (1,4);
\draw  (1,3) rectangle (2,4);
\draw  (2,3) rectangle (3,4);
\draw  (3,3) rectangle (4,4);
\draw  (4,3) rectangle (5,4);
\draw  (5,3) rectangle (6,4);
\draw  (6,3) rectangle (7,4);
\draw  (-1,4) rectangle (0,5);
\draw  (0,4) rectangle (1,5);
\draw  (1,4) rectangle (2,5);
\draw  (2,4) rectangle (3,5);
\draw  (3,4) rectangle (4,5);
\draw  (4,4) rectangle (5,5);
\draw  (5,4) rectangle (6,5);
\draw  (6,4) rectangle (7,5);
\draw  (-1,5) rectangle (0,6);
\draw  (0,5) rectangle (1,6);
\draw  (1,5) rectangle (2,6);
\draw  (2,5) rectangle (3,6);
\draw  (3,5) rectangle (4,6);
\draw  (4,5) rectangle (5,6);
\draw  (5,5) rectangle (6,6);
\draw  (6,5) rectangle (7,6);
\draw [ fill={rgb,255:red,0; green,0; blue,0}, fill opacity=1] (-1,0) rectangle (0,1);
\end{circuitikz}
}%
\caption{Representation of the chocolate bar $\SCG(7, 5)$}
\label{fig:rcg57}
\end{figure}
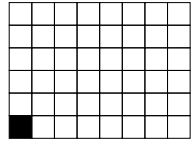

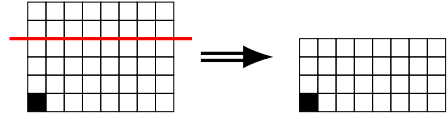
\begin{figure}[!ht]
\centering
\resizebox{1\textwidth/2}{!}{%
\begin{circuitikz}
\tikzstyle{every node}=[font=\fontsize{18.2pt}{23.7pt}\selectfont]
\draw [color=white]  (-2,-1) rectangle (8,7);
\draw  (0,0) rectangle (1,1);
\draw  (1,0) rectangle (2,1);
\draw  (2,0) rectangle (3,1);
\draw  (3,0) rectangle (4,1);
\draw  (4,0) rectangle (5,1);
\draw  (5,0) rectangle (6,1);
\draw  (6,0) rectangle (7,1);
\draw  (-1,1) rectangle (0,2);
\draw  (0,1) rectangle (1,2);
\draw  (1,1) rectangle (2,2);
\draw  (2,1) rectangle (3,2);
\draw  (3,1) rectangle (4,2);
\draw  (4,1) rectangle (5,2);
\draw  (5,1) rectangle (6,2);
\draw  (6,1) rectangle (7,2);
\draw  (-1,2) rectangle (0,3);
\draw  (0,2) rectangle (1,3);
\draw  (1,2) rectangle (2,3);
\draw  (2,2) rectangle (3,3);
\draw  (3,2) rectangle (4,3);
\draw  (4,2) rectangle (5,3);
\draw  (5,2) rectangle (6,3);
\draw  (6,2) rectangle (7,3);
\draw  (-1,3) rectangle (0,4);
\draw  (0,3) rectangle (1,4);
\draw  (1,3) rectangle (2,4);
\draw  (2,3) rectangle (3,4);
\draw  (3,3) rectangle (4,4);
\draw  (4,3) rectangle (5,4);
\draw  (5,3) rectangle (6,4);
\draw  (6,3) rectangle (7,4);
\draw  (-1,4) rectangle (0,5);
\draw  (0,4) rectangle (1,5);
\draw  (1,4) rectangle (2,5);
\draw  (2,4) rectangle (3,5);
\draw  (3,4) rectangle (4,5);
\draw  (4,4) rectangle (5,5);
\draw  (5,4) rectangle (6,5);
\draw  (6,4) rectangle (7,5);
\draw  (-1,5) rectangle (0,6);
\draw  (0,5) rectangle (1,6);
\draw  (1,5) rectangle (2,6);
\draw  (2,5) rectangle (3,6);
\draw  (3,5) rectangle (4,6);
\draw  (4,5) rectangle (5,6);
\draw  (5,5) rectangle (6,6);
\draw  (6,5) rectangle (7,6);
\draw [ fill={rgb,255:red,0; green,0; blue,0}, fill opacity=1] (-1,0) rectangle (0,1);
\draw [line width=5,color=red] (-2,4) -- (8,4);
\end{circuitikz}
\begin{circuitikz}
\draw [color=white]  (-2,-4) rectangle (2,4);
\draw [line width=5, double distance=6,
             arrows = {-Latex[length=0pt 3 0]}] (-2,0) -- (2,0);
\end{circuitikz}
\begin{circuitikz}
\tikzstyle{every node}=[font=\fontsize{18.2pt}{23.7pt}\selectfont]
\draw [color=white]  (-2,-1) rectangle (8,7);
\draw  (0,0) rectangle (1,1);
\draw  (1,0) rectangle (2,1);
\draw  (2,0) rectangle (3,1);
\draw  (3,0) rectangle (4,1);
\draw  (4,0) rectangle (5,1);
\draw  (5,0) rectangle (6,1);
\draw  (6,0) rectangle (7,1);
\draw  (-1,1) rectangle (0,2);
\draw  (0,1) rectangle (1,2);
\draw  (1,1) rectangle (2,2);
\draw  (2,1) rectangle (3,2);
\draw  (3,1) rectangle (4,2);
\draw  (4,1) rectangle (5,2);
\draw  (5,1) rectangle (6,2);
\draw  (6,1) rectangle (7,2);
\draw  (-1,2) rectangle (0,3);
\draw  (0,2) rectangle (1,3);
\draw  (1,2) rectangle (2,3);
\draw  (2,2) rectangle (3,3);
\draw  (3,2) rectangle (4,3);
\draw  (4,2) rectangle (5,3);
\draw  (5,2) rectangle (6,3);
\draw  (6,2) rectangle (7,3);
\draw  (-1,3) rectangle (0,4);
\draw  (0,3) rectangle (1,4);
\draw  (1,3) rectangle (2,4);
\draw  (2,3) rectangle (3,4);
\draw  (3,3) rectangle (4,4);
\draw  (4,3) rectangle (5,4);
\draw  (5,3) rectangle (6,4);
\draw  (6,3) rectangle (7,4);
\draw [ fill={rgb,255:red,0; green,0; blue,0}, fill opacity=1] (-1,0) rectangle (0,1);
\end{circuitikz}
}%
\caption{Representation of a move from $\SCG(7, 5)$}
\label{fig:hcutrcg}
\end{figure}

\begin{figure}[!ht]
\centering
\resizebox{1\textwidth/2}{!}{%
\begin{circuitikz}
\tikzstyle{every node}=[font=\fontsize{18.2pt}{23.7pt}\selectfont]
\draw [color=white]  (-2,-1) rectangle (8,7);
\draw  (0,0) rectangle (1,1);
\draw  (1,0) rectangle (2,1);
\draw  (2,0) rectangle (3,1);
\draw  (3,0) rectangle (4,1);
\draw  (4,0) rectangle (5,1);
\draw  (5,0) rectangle (6,1);
\draw  (6,0) rectangle (7,1);
\draw  (-1,1) rectangle (0,2);
\draw  (0,1) rectangle (1,2);
\draw  (1,1) rectangle (2,2);
\draw  (2,1) rectangle (3,2);
\draw  (3,1) rectangle (4,2);
\draw  (4,1) rectangle (5,2);
\draw  (5,1) rectangle (6,2);
\draw  (6,1) rectangle (7,2);
\draw  (-1,2) rectangle (0,3);
\draw  (0,2) rectangle (1,3);
\draw  (1,2) rectangle (2,3);
\draw  (2,2) rectangle (3,3);
\draw  (3,2) rectangle (4,3);
\draw  (4,2) rectangle (5,3);
\draw  (5,2) rectangle (6,3);
\draw  (6,2) rectangle (7,3);
\draw  (-1,3) rectangle (0,4);
\draw  (0,3) rectangle (1,4);
\draw  (1,3) rectangle (2,4);
\draw  (2,3) rectangle (3,4);
\draw  (3,3) rectangle (4,4);
\draw  (4,3) rectangle (5,4);
\draw  (5,3) rectangle (6,4);
\draw  (6,3) rectangle (7,4);
\draw  (-1,4) rectangle (0,5);
\draw  (0,4) rectangle (1,5);
\draw  (1,4) rectangle (2,5);
\draw  (2,4) rectangle (3,5);
\draw  (3,4) rectangle (4,5);
\draw  (4,4) rectangle (5,5);
\draw  (5,4) rectangle (6,5);
\draw  (6,4) rectangle (7,5);
\draw  (-1,5) rectangle (0,6);
\draw  (0,5) rectangle (1,6);
\draw  (1,5) rectangle (2,6);
\draw  (2,5) rectangle (3,6);
\draw  (3,5) rectangle (4,6);
\draw  (4,5) rectangle (5,6);
\draw  (5,5) rectangle (6,6);
\draw  (6,5) rectangle (7,6);
\draw [ fill={rgb,255:red,0; green,0; blue,0}, fill opacity=1] (-1,0) rectangle (0,1);
\draw [line width=5,color=red] (3,-1) -- (3,7);
\end{circuitikz}
\begin{circuitikz}
\draw [color=white]  (-2,-4) rectangle (2,4);
\draw [line width=5, double distance=6,
             arrows = {-Latex[length=0pt 3 0]}] (-2,0) -- (2,0);
\end{circuitikz}
\begin{circuitikz}
\tikzstyle{every node}=[font=\fontsize{18.2pt}{23.7pt}\selectfont]
\draw [color=white]  (-2,-1) rectangle (8,7);
\draw  (0,0) rectangle (1,1);
\draw  (1,0) rectangle (2,1);
\draw  (2,0) rectangle (3,1);
\draw  (-1,1) rectangle (0,2);
\draw  (0,1) rectangle (1,2);
\draw  (1,1) rectangle (2,2);
\draw  (2,1) rectangle (3,2);
\draw  (-1,2) rectangle (0,3);
\draw  (0,2) rectangle (1,3);
\draw  (1,2) rectangle (2,3);
\draw  (2,2) rectangle (3,3);
\draw  (-1,3) rectangle (0,4);
\draw  (0,3) rectangle (1,4);
\draw  (1,3) rectangle (2,4);
\draw  (2,3) rectangle (3,4);
\draw  (-1,4) rectangle (0,5);
\draw  (0,4) rectangle (1,5);
\draw  (1,4) rectangle (2,5);
\draw  (2,4) rectangle (3,5);
\draw  (-1,5) rectangle (0,6);
\draw  (0,5) rectangle (1,6);
\draw  (1,5) rectangle (2,6);
\draw  (2,5) rectangle (3,6);
\draw [ fill={rgb,255:red,0; green,0; blue,0}, fill opacity=1] (-1,0) rectangle (0,1);
\end{circuitikz}
}%
\caption{Representation of a move from $\SCG(5, 7)$}
\label{fig:vcutrcg}
\end{figure}
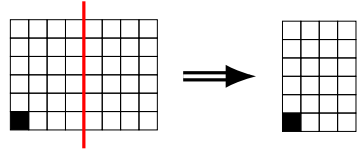

Moreover, generalized Chocolate Games in which the chocolate bar has an increasing staircase shape have also been studied. Figure \ref{fig:icg} is an example of such a chocolate bar. This chocolate bar is a part of $(7+1)\times(3+1)$ chocolate bar and  the height of each column is bounded by a function $f(i) = \left \lfloor \frac{i}{2} \right \rfloor$. We call such a chocolate bar $\CB(f, 7, 3)$; a precise definition will be given later. 
Figures \ref{fig:hcuticg} and \ref{fig:vcuticg} show two legal moves from this chocolate bar.

\begin{figure}[!ht]
\centering
\resizebox{1\textwidth/4}{!}{%
\begin{circuitikz}
\tikzstyle{every node}=[font=\fontsize{18.2pt}{23.7pt}\selectfont]
\draw [color=white]  (-2,-1) rectangle (8,7);
\draw  (0,0) rectangle (1,1);
\draw  (1,0) rectangle (2,1);
\draw  (2,0) rectangle (3,1);
\draw  (3,0) rectangle (4,1);
\draw  (4,0) rectangle (5,1);
\draw  (5,0) rectangle (6,1);
\draw  (6,0) rectangle (7,1);
\draw  (1,1) rectangle (2,2);
\draw  (2,1) rectangle (3,2);
\draw  (3,1) rectangle (4,2);
\draw  (4,1) rectangle (5,2);
\draw  (5,1) rectangle (6,2);
\draw  (6,1) rectangle (7,2);
\draw  (3,2) rectangle (4,3);
\draw  (4,2) rectangle (5,3);
\draw  (5,2) rectangle (6,3);
\draw  (6,2) rectangle (7,3);
\draw  (5,3) rectangle (6,4);
\draw  (6,3) rectangle (7,4);
\draw [ fill={rgb,255:red,0; green,0; blue,0}, fill opacity=1] (-1,0) rectangle (0,1);
\end{circuitikz}
}%
\caption{Representation of the chocolate bar $\CB(f, 7, 3)$ with $f(i)=\lfloor \frac{i}{2} \rfloor$ }
\label{fig:icg}
\end{figure}

\begin{figure}[!ht]
\centering
\resizebox{1\textwidth/2}{!}{%
\begin{circuitikz}
\tikzstyle{every node}=[font=\fontsize{18.2pt}{23.7pt}\selectfont]
\draw [color=white]  (-2,-1) rectangle (8,7);
\draw  (0,0) rectangle (1,1);
\draw  (1,0) rectangle (2,1);
\draw  (2,0) rectangle (3,1);
\draw  (3,0) rectangle (4,1);
\draw  (4,0) rectangle (5,1);
\draw  (5,0) rectangle (6,1);
\draw  (6,0) rectangle (7,1);
\draw  (1,1) rectangle (2,2);
\draw  (2,1) rectangle (3,2);
\draw  (3,1) rectangle (4,2);
\draw  (4,1) rectangle (5,2);
\draw  (5,1) rectangle (6,2);
\draw  (6,1) rectangle (7,2);
\draw  (3,2) rectangle (4,3);
\draw  (4,2) rectangle (5,3);
\draw  (5,2) rectangle (6,3);
\draw  (6,2) rectangle (7,3);
\draw  (5,3) rectangle (6,4);
\draw  (6,3) rectangle (7,4);
\draw [ fill={rgb,255:red,0; green,0; blue,0}, fill opacity=1] (-1,0) rectangle (0,1);
\draw [line width=5,color=red] (-2,3) -- (8,3);
\end{circuitikz}
\begin{circuitikz}
\draw [color=white]  (-2,-4) rectangle (2,4);
\draw [line width=5, double distance=6,
             arrows = {-Latex[length=0pt 3 0]}] (-2,0) -- (2,0);
\end{circuitikz}
\begin{circuitikz}
\tikzstyle{every node}=[font=\fontsize{18.2pt}{23.7pt}\selectfont]
\draw [color=white]  (-2,-1) rectangle (8,7);
\draw  (0,0) rectangle (1,1);
\draw  (1,0) rectangle (2,1);
\draw  (2,0) rectangle (3,1);
\draw  (3,0) rectangle (4,1);
\draw  (4,0) rectangle (5,1);
\draw  (5,0) rectangle (6,1);
\draw  (6,0) rectangle (7,1);
\draw  (1,1) rectangle (2,2);
\draw  (2,1) rectangle (3,2);
\draw  (3,1) rectangle (4,2);
\draw  (4,1) rectangle (5,2);
\draw  (5,1) rectangle (6,2);
\draw  (6,1) rectangle (7,2);
\draw  (3,2) rectangle (4,3);
\draw  (4,2) rectangle (5,3);
\draw  (5,2) rectangle (6,3);
\draw  (6,2) rectangle (7,3);
\draw [ fill={rgb,255:red,0; green,0; blue,0}, fill opacity=1] (-1,0) rectangle (0,1);
\end{circuitikz}
}%
\caption{Representation of a move from $\CB(f, 7, 3)$ with $f(i)=\lfloor \frac{i}{2} \rfloor$}
\label{fig:hcuticg}
\end{figure}

\begin{figure}[!ht]
\centering
\resizebox{1\textwidth/2}{!}{%
\begin{circuitikz}
\tikzstyle{every node}=[font=\fontsize{18.2pt}{23.7pt}\selectfont]
\draw [color=white]  (-2,-1) rectangle (8,7);
\draw  (0,0) rectangle (1,1);
\draw  (1,0) rectangle (2,1);
\draw  (2,0) rectangle (3,1);
\draw  (3,0) rectangle (4,1);
\draw  (4,0) rectangle (5,1);
\draw  (5,0) rectangle (6,1);
\draw  (6,0) rectangle (7,1);
\draw  (1,1) rectangle (2,2);
\draw  (2,1) rectangle (3,2);
\draw  (3,1) rectangle (4,2);
\draw  (4,1) rectangle (5,2);
\draw  (5,1) rectangle (6,2);
\draw  (6,1) rectangle (7,2);
\draw  (3,2) rectangle (4,3);
\draw  (4,2) rectangle (5,3);
\draw  (5,2) rectangle (6,3);
\draw  (6,2) rectangle (7,3);
\draw  (5,3) rectangle (6,4);
\draw  (6,3) rectangle (7,4);
\draw [ fill={rgb,255:red,0; green,0; blue,0}, fill opacity=1] (-1,0) rectangle (0,1);
\draw [line width=5,color=red] (3,-1) -- (3,7);
\end{circuitikz}
\begin{circuitikz}
\draw [color=white]  (-2,-4) rectangle (2,4);
\draw [line width=5, double distance=6,
             arrows = {-Latex[length=0pt 3 0]}] (-2,0) -- (2,0);
\end{circuitikz}
\begin{circuitikz}
\tikzstyle{every node}=[font=\fontsize{18.2pt}{23.7pt}\selectfont]
\draw [color=white]  (-2,-1) rectangle (8,7);
\draw  (0,0) rectangle (1,1);
\draw  (1,0) rectangle (2,1);
\draw  (2,0) rectangle (3,1);
\draw  (1,1) rectangle (2,2);
\draw  (2,1) rectangle (3,2);
\draw [ fill={rgb,255:red,0; green,0; blue,0}, fill opacity=1] (-1,0) rectangle (0,1);
\end{circuitikz}
}%
\caption{Representation of a move from $\CB(f, 7, 3)$ with $f(i)=\lfloor \frac{i}{2} \rfloor$}
\label{fig:vcuticg}
\end{figure}

We can use such Chocolate Games to express ``Nim with a pass'' under the standard normal play convention. A game with a pass is a game in which one of the players may pass exactly once during the play of the game. Once a pass has been made, neither player may make a pass again. In addition, a pass cannot be made from a terminal position.
For example, when we play one-pile Nim with a pass, one can choose to remove an arbitrary positive number of tokens from the pile or to make a pass on their turn if a pass has not been used, but cannot make any move from the terminal position. 

The Chocolate Game played on the chocolate bar shown in Fig. \ref{fig:passchoco} is isomorphic to one-pile Nim with a pass. Since one can choose to reduce the width of the bar as long as the width is greater than one, or to reduce the height from two to one if no one has reduced it yet, but cannot make any move if the width has been reduced to one since at the same time, even if the height has not been reduced, it is reduced to one and only the bitter block remains. In \cite{MN21}, three-dimensional Chocolate Games, like shown in Fig. \ref{fig:3d_passchoco}, are also considered as a generalization of two-pile Nim with a pass.

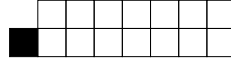
\begin{figure}[!ht]
\centering
\resizebox{1\textwidth/4}{!}{%
\begin{circuitikz}
\tikzstyle{every node}=[font=\fontsize{18.2pt}{23.7pt}\selectfont]
\draw  (0,0) rectangle (1,1);
\draw  (1,0) rectangle (2,1);
\draw  (2,0) rectangle (3,1);
\draw  (3,0) rectangle (4,1);
\draw  (4,0) rectangle (5,1);
\draw  (5,0) rectangle (6,1);
\draw  (6,0) rectangle (7,1);
\draw  (0,1) rectangle (1,2);
\draw  (1,1) rectangle (2,2);
\draw  (2,1) rectangle (3,2);
\draw  (3,1) rectangle (4,2);
\draw  (4,1) rectangle (5,2);
\draw  (5,1) rectangle (6,2);
\draw  (6,1) rectangle (7,2);
\draw [ fill={rgb,255:red,0; green,0; blue,0}, fill opacity=1] (-1,0) rectangle (0,1);
\end{circuitikz}
}%
\caption{Representation of a chocolate bar on which Chocolate Game is isomorphic to one-pile Nim with a pass}
\label{fig:passchoco}
\end{figure}

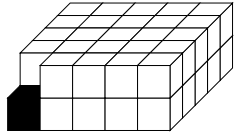
\begin{figure}[!ht]
    \centering
\resizebox{1\textwidth/4}{!}{%
\begin{tikzpicture}
	\draw [ fill={rgb,255:red,255; green,255; blue,255}, fill opacity=1] (0,1,0) -- (1,1,0) -- (1,1,1) -- (0,1,1) -- cycle;

	\draw [ fill={rgb,255:red,255; green,255; blue,255}, fill opacity=1] (1,1,0) -- (2,1,0) -- (2,1,1) -- (1,1,1) -- cycle;

	\draw [ fill={rgb,255:red,255; green,255; blue,255}, fill opacity=1] (2,1,0) -- (3,1,0) -- (3,1,1) -- (2,1,1) -- cycle;

	\draw[ fill={rgb,255:red,255; green,255; blue,255}, fill opacity=1] (3,1,0) -- (4,1,0) -- (4,1,1) -- (3,1,1) -- cycle;

    \draw[ fill={rgb,255:red,255; green,255; blue,255}, fill opacity=1] (5,0,0) -- (5,1,0) -- (5,1,1) -- (5,0,1) -- cycle;
	\draw[ fill={rgb,255:red,255; green,255; blue,255}, fill opacity=1] (4,1,0) -- (5,1,0) -- (5,1,1) -- (4,1,1) -- cycle;

	\draw[ fill={rgb,255:red,255; green,255; blue,255}, fill opacity=1] (0,1,1) -- (1,1,1) -- (1,1,2) -- (0,1,2) -- cycle;

	\draw[ fill={rgb,255:red,255; green,255; blue,255}, fill opacity=1] (1,1,1) -- (2,1,1) -- (2,1,2) -- (1,1,2) -- cycle;

	\draw[ fill={rgb,255:red,255; green,255; blue,255}, fill opacity=1] (2,1,1) -- (3,1,1) -- (3,1,2) -- (2,1,2) -- cycle;

	\draw[ fill={rgb,255:red,255; green,255; blue,255}, fill opacity=1] (3,1,1) -- (4,1,1) -- (4,1,2) -- (3,1,2) -- cycle;

    \draw[ fill={rgb,255:red,255; green,255; blue,255}, fill opacity=1] (5,0,1) -- (5,1,1) -- (5,1,2) -- (5,0,2) -- cycle;
	\draw[ fill={rgb,255:red,255; green,255; blue,255}, fill opacity=1] (4,1,1) -- (5,1,1) -- (5,1,2) -- (4,1,2) -- cycle;

	\draw[ fill={rgb,255:red,255; green,255; blue,255}, fill opacity=1] (0,1,2) -- (1,1,2) -- (1,1,3) -- (0,1,3) -- cycle;

	\draw[ fill={rgb,255:red,255; green,255; blue,255}, fill opacity=1] (1,1,2) -- (2,1,2) -- (2,1,3) -- (1,1,3) -- cycle;

	\draw[ fill={rgb,255:red,255; green,255; blue,255}, fill opacity=1] (2,1,2) -- (3,1,2) -- (3,1,3) -- (2,1,3) -- cycle;

	\draw[ fill={rgb,255:red,255; green,255; blue,255}, fill opacity=1] (3,1,2) -- (4,1,2) -- (4,1,3) -- (3,1,3) -- cycle;

    \draw[ fill={rgb,255:red,255; green,255; blue,255}, fill opacity=1] (5,0,2) -- (5,1,2) -- (5,1,3) -- (5,0,3) -- cycle;
	\draw[ fill={rgb,255:red,255; green,255; blue,255}, fill opacity=1] (4,1,2) -- (5,1,2) -- (5,1,3) -- (4,1,3) -- cycle;

	\draw[ fill={rgb,255:red,255; green,255; blue,255}, fill opacity=1] (0,1,3) -- (1,1,3) -- (1,1,4) -- (0,1,4) -- cycle;

	\draw[ fill={rgb,255:red,255; green,255; blue,255}, fill opacity=1] (1,1,3) -- (2,1,3) -- (2,1,4) -- (1,1,4) -- cycle;

	\draw[ fill={rgb,255:red,255; green,255; blue,255}, fill opacity=1] (2,1,3) -- (3,1,3) -- (3,1,4) -- (2,1,4) -- cycle;

	\draw[ fill={rgb,255:red,255; green,255; blue,255}, fill opacity=1] (3,1,3) -- (4,1,3) -- (4,1,4) -- (3,1,4) -- cycle;

    \draw[ fill={rgb,255:red,255; green,255; blue,255}, fill opacity=1] (5,0,3) -- (5,1,3) -- (5,1,4) -- (5,0,4) -- cycle;
	\draw[ fill={rgb,255:red,255; green,255; blue,255}, fill opacity=1] (4,1,3) -- (5,1,3) -- (5,1,4) -- (4,1,4) -- cycle;

	\draw[ fill={rgb,255:red,0; green,0; blue,0}, fill opacity=1] (0,1,4) -- (1,1,4) -- (1,1,5) -- (0,1,5) -- cycle;
	\draw[ fill={rgb,255:red,0; green,0; blue,0}, fill opacity=1] (0,0,5) -- (0,1,5) -- (1,1,5) -- (1,0,5) -- cycle;

	\draw[ fill={rgb,255:red,255; green,255; blue,255}, fill opacity=1] (1,1,4) -- (2,1,4) -- (2,1,5) -- (1,1,5) -- cycle;
	\draw[ fill={rgb,255:red,255; green,255; blue,255}, fill opacity=1] (1,0,5) -- (1,1,5) -- (2,1,5) -- (2,0,5) -- cycle;

	\draw[ fill={rgb,255:red,255; green,255; blue,255}, fill opacity=1] (2,1,4) -- (3,1,4) -- (3,1,5) -- (2,1,5) -- cycle;
	\draw[ fill={rgb,255:red,255; green,255; blue,255}, fill opacity=1] (2,0,5) -- (2,1,5) -- (3,1,5) -- (3,0,5) -- cycle;

	\draw[ fill={rgb,255:red,255; green,255; blue,255}, fill opacity=1] (3,1,4) -- (4,1,4) -- (4,1,5) -- (3,1,5) -- cycle;
	\draw[ fill={rgb,255:red,255; green,255; blue,255}, fill opacity=1] (3,0,5) -- (3,1,5) -- (4,1,5) -- (4,0,5) -- cycle;

    \draw[ fill={rgb,255:red,255; green,255; blue,255}, fill opacity=1] (5,0,4) -- (5,1,4) -- (5,1,5) -- (5,0,5) -- cycle;
	\draw[ fill={rgb,255:red,255; green,255; blue,255}, fill opacity=1] (4,1,4) -- (5,1,4) -- (5,1,5) -- (4,1,5) -- cycle;
	\draw[ fill={rgb,255:red,255; green,255; blue,255}, fill opacity=1] (4,0,5) -- (4,1,5) -- (5,1,5) -- (5,0,5) -- cycle;

	\draw[ fill={rgb,255:red,255; green,255; blue,255}, fill opacity=1](0,2,0) -- (1,2,0) -- (1,2,1) -- (0,2,1) -- cycle;

	\draw[ fill={rgb,255:red,255; green,255; blue,255}, fill opacity=1](1,2,0) -- (2,2,0) -- (2,2,1) -- (1,2,1) -- cycle;

	\draw [ fill={rgb,255:red,255; green,255; blue,255}, fill opacity=1] (2,2,0) -- (3,2,0) -- (3,2,1) -- (2,2,1) -- cycle;

	\draw[ fill={rgb,255:red,255; green,255; blue,255}, fill opacity=1] (3,2,0) -- (4,2,0) -- (4,2,1) -- (3,2,1) -- cycle;

    \draw[ fill={rgb,255:red,255; green,255; blue,255}, fill opacity=1] (5,1,0) -- (5,2,0) -- (5,2,1) -- (5,1,1) -- cycle;
	\draw[ fill={rgb,255:red,255; green,255; blue,255}, fill opacity=1] (4,2,0) -- (5,2,0) -- (5,2,1) -- (4,2,1) -- cycle;

	\draw[ fill={rgb,255:red,255; green,255; blue,255}, fill opacity=1] (0,2,1) -- (2,2,1) -- (2,2,2) -- (0,2,2) -- cycle;
    
	\draw[ fill={rgb,255:red,255; green,255; blue,255}, fill opacity=1] (1,2,1) -- (2,2,1) -- (2,2,2) -- (1,2,2) -- cycle;
	\draw[ fill={rgb,255:red,255; green,255; blue,255}, fill opacity=1] (1,1,2) -- (1,2,2) -- (2,2,2) -- (2,1,2) -- cycle;

	\draw[ fill={rgb,255:red,255; green,255; blue,255}, fill opacity=1] (2,2,1) -- (3,2,1) -- (3,2,2) -- (2,2,2) -- cycle;

	\draw[ fill={rgb,255:red,255; green,255; blue,255}, fill opacity=1] (3,2,1) -- (4,2,1) -- (4,2,2) -- (3,2,2) -- cycle;

    \draw[ fill={rgb,255:red,255; green,255; blue,255}, fill opacity=1] (5,1,1) -- (5,2,1) -- (5,2,2) -- (5,1,2) -- cycle;
	\draw[ fill={rgb,255:red,255; green,255; blue,255}, fill opacity=1] (4,2,1) -- (5,2,1) -- (5,2,2) -- (4,2,2) -- cycle;

	\draw[ fill={rgb,255:red,255; green,255; blue,255}, fill opacity=1] (0,2,2) -- (1,2,2) -- (1,2,3) -- (0,2,3) -- cycle;
    
	\draw[ fill={rgb,255:red,255; green,255; blue,255}, fill opacity=1] (1,2,2) -- (2,2,2) -- (2,2,3) -- (1,2,3) -- cycle;

	\draw[ fill={rgb,255:red,255; green,255; blue,255}, fill opacity=1] (2,2,2) -- (3,2,2) -- (3,2,3) -- (2,2,3) -- cycle;

	\draw[ fill={rgb,255:red,255; green,255; blue,255}, fill opacity=1] (3,2,2) -- (4,2,2) -- (4,2,3) -- (3,2,3) -- cycle;

    \draw[ fill={rgb,255:red,255; green,255; blue,255}, fill opacity=1] (5,1,2) -- (5,2,2) -- (5,2,3) -- (5,1,3) -- cycle;
	\draw[ fill={rgb,255:red,255; green,255; blue,255}, fill opacity=1] (4,2,2) -- (5,2,2) -- (5,2,3) -- (4,2,3) -- cycle;

	\draw[ fill={rgb,255:red,255; green,255; blue,255}, fill opacity=1] (0,2,3) -- (1,2,3) -- (1,2,4) -- (0,2,4) -- cycle;
	\draw[ fill={rgb,255:red,255; green,255; blue,255}, fill opacity=1] (0,1,4) -- (0,2,4) -- (1,2,4) -- (1,1,4) -- cycle;

	\draw[ fill={rgb,255:red,255; green,255; blue,255}, fill opacity=1] (1,2,3) -- (2,2,3) -- (2,2,4) -- (1,2,4) -- cycle;
	\draw[ fill={rgb,255:red,255; green,255; blue,255}, fill opacity=1] (1,1,4) -- (1,2,4) -- (2,2,4) -- (2,1,4) -- cycle;

	\draw[ fill={rgb,255:red,255; green,255; blue,255}, fill opacity=1] (2,2,3) -- (3,2,3) -- (3,2,4) -- (2,2,4) -- cycle;
	\draw[ fill={rgb,255:red,255; green,255; blue,255}, fill opacity=1] (2,1,4) -- (2,2,4) -- (3,2,4) -- (3,1,4) -- cycle;

	\draw[ fill={rgb,255:red,255; green,255; blue,255}, fill opacity=1] (3,2,3) -- (4,2,3) -- (4,2,4) -- (3,2,4) -- cycle;

    \draw[ fill={rgb,255:red,255; green,255; blue,255}, fill opacity=1] (5,1,3) -- (5,2,3) -- (5,2,4) -- (5,1,4) -- cycle;
	\draw[ fill={rgb,255:red,255; green,255; blue,255}, fill opacity=1] (4,2,3) -- (5,2,3) -- (5,2,4) -- (4,2,4) -- cycle;

	\draw[ fill={rgb,255:red,255; green,255; blue,255}, fill opacity=1] (1,2,4) -- (2,2,4) -- (2,2,5) -- (1,2,5) -- cycle;
	\draw[ fill={rgb,255:red,255; green,255; blue,255}, fill opacity=1] (1,1,5) -- (1,2,5) -- (2,2,5) -- (2,1,5) -- cycle;

	\draw[ fill={rgb,255:red,255; green,255; blue,255}, fill opacity=1] (2,2,4) -- (3,2,4) -- (3,2,5) -- (2,2,5) -- cycle;
	\draw[ fill={rgb,255:red,255; green,255; blue,255}, fill opacity=1] (2,1,5) -- (2,2,5) -- (3,2,5) -- (3,1,5) -- cycle;

	\draw[ fill={rgb,255:red,255; green,255; blue,255}, fill opacity=1] (3,2,4) -- (4,2,4) -- (4,2,5) -- (3,2,5) -- cycle;
	\draw[ fill={rgb,255:red,255; green,255; blue,255}, fill opacity=1] (3,1,5) -- (3,2,5) -- (4,2,5) -- (4,1,5) -- cycle;

    \draw[ fill={rgb,255:red,255; green,255; blue,255}, fill opacity=1] (5,1,4) -- (5,2,4) -- (5,2,5) -- (5,1,5) -- cycle;
	\draw[ fill={rgb,255:red,255; green,255; blue,255}, fill opacity=1] (4,2,4) -- (5,2,4) -- (5,2,5) -- (4,2,5) -- cycle;
	\draw[ fill={rgb,255:red,255; green,255; blue,255}, fill opacity=1] (4,1,5) -- (4,2,5) -- (5,2,5) -- (5,1,5) -- cycle;

\end{tikzpicture}
}%
    \caption{Representation of a three--dimensional chocolate bar on which Chocolate Game is isomorphic to two--pile Nim with a pass}
    \label{fig:3d_passchoco}
\end{figure}

Considering them, Chocolate Games on increasing staircase chocolate bars can be considered as a generalization of Nim with a pass.

Inspired by there results, in this paper, we consider Chocolate Games on decreasing staircase chocolate bars like shown as Fig. \ref{fig:dcg}. 
This chocolate bar is a part of $(7+1)\times(4+1)$ chocolate bar and  the height of each column is bounded by a function $f(i) = \max(\lfloor \frac{8-i}{2} \rfloor,0)$. We call such a chocolate bar $\DCG(f, 7, 4)$; a precise definition will be given later.

\begin{figure}[!ht]
\centering
\resizebox{1\textwidth/4}{!}{%
\begin{circuitikz}
\tikzstyle{every node}=[font=\fontsize{18.2pt}{23.7pt}\selectfont]
\draw [color=white]  (-2,-1) rectangle (8,7);
\draw  (0,0) rectangle (1,1);
\draw  (1,0) rectangle (2,1);
\draw  (2,0) rectangle (3,1);
\draw  (3,0) rectangle (4,1);
\draw  (4,0) rectangle (5,1);
\draw  (5,0) rectangle (6,1);
\draw  (6,0) rectangle (7,1);
\draw  (-1,1) rectangle (0,2);
\draw  (0,1) rectangle (1,2);
\draw  (1,1) rectangle (2,2);
\draw  (2,1) rectangle (3,2);
\draw  (3,1) rectangle (4,2);
\draw  (4,1) rectangle (5,2);
\draw  (5,1) rectangle (6,2);
\draw  (-1,2) rectangle (0,3);
\draw  (0,2) rectangle (1,3);
\draw  (1,2) rectangle (2,3);
\draw  (2,2) rectangle (3,3);
\draw  (3,2) rectangle (4,3);
\draw  (-1,3) rectangle (0,4);
\draw  (0,3) rectangle (1,4);
\draw  (1,3) rectangle (2,4);
\draw  (-1,4) rectangle (0,5);
\draw [ fill={rgb,255:red,0; green,0; blue,0}, fill opacity=1] (-1,0) rectangle (0,1);
\end{circuitikz}
}%
\caption{Representation of the chocolate bar $\DCG(f, 7, 4)$ with $f(i)=\max(\lfloor \frac{8-i}{2} \rfloor,0)$ }
\label{fig:dcg}
\end{figure}
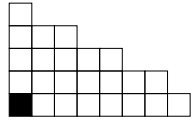


For decreasing staircase chocolate bars, we impose the rule that at least one square must be removed from the rightmost column on each turn. Note that, for increasing chocolate bars, at least one square is always removed from the rightmost column on each turn even without this additional rule. 
Figures \ref{fig:vcutdcg} and \ref{fig:hcutdcg} show legal moves and Fig. \ref{fig:illegalcutdcg} shows an illegal move since no square on the rightmost column is removed.

\begin{figure}[!ht]
\centering
\resizebox{1\textwidth/2}{!}{%
\begin{circuitikz}
\tikzstyle{every node}=[font=\fontsize{18.2pt}{23.7pt}\selectfont]
\draw [color=white]  (-2,-1) rectangle (8,7);
\draw  (0,0) rectangle (1,1);
\draw  (1,0) rectangle (2,1);
\draw  (2,0) rectangle (3,1);
\draw  (3,0) rectangle (4,1);
\draw  (4,0) rectangle (5,1);
\draw  (5,0) rectangle (6,1);
\draw  (6,0) rectangle (7,1);
\draw  (-1,1) rectangle (0,2);
\draw  (0,1) rectangle (1,2);
\draw  (1,1) rectangle (2,2);
\draw  (2,1) rectangle (3,2);
\draw  (3,1) rectangle (4,2);
\draw  (4,1) rectangle (5,2);
\draw  (5,1) rectangle (6,2);
\draw  (-1,2) rectangle (0,3);
\draw  (0,2) rectangle (1,3);
\draw  (1,2) rectangle (2,3);
\draw  (2,2) rectangle (3,3);
\draw  (3,2) rectangle (4,3);
\draw  (-1,3) rectangle (0,4);
\draw  (0,3) rectangle (1,4);
\draw  (1,3) rectangle (2,4);
\draw  (-1,4) rectangle (0,5);
\draw [ fill={rgb,255:red,0; green,0; blue,0}, fill opacity=1] (-1,0) rectangle (0,1);
\draw [line width=5,color=red] (3,-1) -- (3,7);
\end{circuitikz}
\begin{circuitikz}
\draw [color=white]  (-2,-4) rectangle (2,4);
\draw [line width=5, double distance=6,
             arrows = {-Latex[length=0pt 3 0]}] (-2,0) -- (2,0);
\end{circuitikz}
\begin{circuitikz}
\tikzstyle{every node}=[font=\fontsize{18.2pt}{23.7pt}\selectfont]
\draw [color=white]  (-2,-1) rectangle (8,7);
\draw  (0,0) rectangle (1,1);
\draw  (1,0) rectangle (2,1);
\draw  (2,0) rectangle (3,1);
\draw  (-1,1) rectangle (0,2);
\draw  (0,1) rectangle (1,2);
\draw  (1,1) rectangle (2,2);
\draw  (2,1) rectangle (3,2);
\draw  (-1,2) rectangle (0,3);
\draw  (0,2) rectangle (1,3);
\draw  (1,2) rectangle (2,3);
\draw  (2,2) rectangle (3,3);
\draw  (-1,3) rectangle (0,4);
\draw  (0,3) rectangle (1,4);
\draw  (1,3) rectangle (2,4);
\draw  (-1,4) rectangle (0,5);
\draw [ fill={rgb,255:red,0; green,0; blue,0}, fill opacity=1] (-1,0) rectangle (0,1);
\end{circuitikz}
}%
\caption{Representation of a move from $\DCG(f, 7, 4)$ with $f(i)=\max(\lfloor \frac{8-i}{2} \rfloor,0)$}
\label{fig:vcutdcg}
\end{figure}
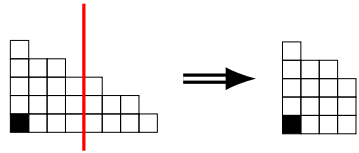

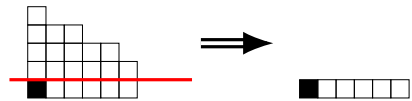
\begin{figure}[!ht]
\centering
\resizebox{1\textwidth/2}{!}{%
\begin{circuitikz}
\tikzstyle{every node}=[font=\fontsize{18.2pt}{23.7pt}\selectfont]
\draw [color=white]  (-2,-1) rectangle (8,7);
\draw  (0,0) rectangle (1,1);
\draw  (1,0) rectangle (2,1);
\draw  (2,0) rectangle (3,1);
\draw  (3,0) rectangle (4,1);
\draw  (4,0) rectangle (5,1);
\draw  (-1,1) rectangle (0,2);
\draw  (0,1) rectangle (1,2);
\draw  (1,1) rectangle (2,2);
\draw  (2,1) rectangle (3,2);
\draw  (3,1) rectangle (4,2);
\draw  (4,1) rectangle (5,2);
\draw  (-1,2) rectangle (0,3);
\draw  (0,2) rectangle (1,3);
\draw  (1,2) rectangle (2,3);
\draw  (2,2) rectangle (3,3);
\draw  (3,2) rectangle (4,3);
\draw  (-1,3) rectangle (0,4);
\draw  (0,3) rectangle (1,4);
\draw  (1,3) rectangle (2,4);
\draw  (-1,4) rectangle (0,5);
\draw [ fill={rgb,255:red,0; green,0; blue,0}, fill opacity=1] (-1,0) rectangle (0,1);
\draw [line width=5,color=red] (-2,1) -- (8,1);
\end{circuitikz}
\begin{circuitikz}
\draw [color=white]  (-2,-4) rectangle (2,4);
\draw [line width=5, double distance=6,
             arrows = {-Latex[length=0pt 3 0]}] (-2,0) -- (2,0);
\end{circuitikz}
\begin{circuitikz}
\tikzstyle{every node}=[font=\fontsize{18.2pt}{23.7pt}\selectfont]
\draw [color=white]  (-2,-1) rectangle (8,7);
\draw  (0,0) rectangle (1,1);
\draw  (1,0) rectangle (2,1);
\draw  (2,0) rectangle (3,1);
\draw  (3,0) rectangle (4,1);
\draw  (4,0) rectangle (5,1);
\draw [ fill={rgb,255:red,0; green,0; blue,0}, fill opacity=1] (-1,0) rectangle (0,1);
\end{circuitikz}
}%
\caption{Representation of a move from $\DCG(f, 5, 4)$ with $f(i)=\max(\lfloor \frac{8-i}{2} \rfloor,0)$}
\label{fig:hcutdcg}
\end{figure}

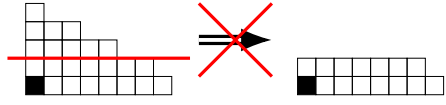
\begin{figure}[!ht]
\centering
\resizebox{1\textwidth/2}{!}{%
\begin{circuitikz}
\tikzstyle{every node}=[font=\fontsize{18.2pt}{23.7pt}\selectfont]
\draw [color=white]  (-2,-1) rectangle (8,7);
\draw  (0,0) rectangle (1,1);
\draw  (1,0) rectangle (2,1);
\draw  (2,0) rectangle (3,1);
\draw  (3,0) rectangle (4,1);
\draw  (4,0) rectangle (5,1);
\draw  (5,0) rectangle (6,1);
\draw  (6,0) rectangle (7,1);
\draw  (-1,1) rectangle (0,2);
\draw  (0,1) rectangle (1,2);
\draw  (1,1) rectangle (2,2);
\draw  (2,1) rectangle (3,2);
\draw  (3,1) rectangle (4,2);
\draw  (4,1) rectangle (5,2);
\draw  (5,1) rectangle (6,2);
\draw  (-1,2) rectangle (0,3);
\draw  (0,2) rectangle (1,3);
\draw  (1,2) rectangle (2,3);
\draw  (2,2) rectangle (3,3);
\draw  (3,2) rectangle (4,3);
\draw  (-1,3) rectangle (0,4);
\draw  (0,3) rectangle (1,4);
\draw  (1,3) rectangle (2,4);
\draw  (-1,4) rectangle (0,5);
\draw [ fill={rgb,255:red,0; green,0; blue,0}, fill opacity=1] (-1,0) rectangle (0,1);
\draw [line width=5,color=red] (-2,2) -- (8,2);
\end{circuitikz}
\begin{circuitikz}
\draw [color=white]  (-2,-4) rectangle (2,4);
\draw [line width=5, double distance=6,
             arrows = {-Latex[length=0pt 3 0]}] (-2,0) -- (2,0);
\draw [line width=5,color=red] (-2,-2) -- (2,2);
\draw [line width=5,color=red] (2,-2) -- (-2,2);
\end{circuitikz}
\begin{circuitikz}
\tikzstyle{every node}=[font=\fontsize{18.2pt}{23.7pt}\selectfont]
\draw [color=white]  (-2,-1) rectangle (8,7);
\draw  (0,0) rectangle (1,1);
\draw  (1,0) rectangle (2,1);
\draw  (2,0) rectangle (3,1);
\draw  (3,0) rectangle (4,1);
\draw  (4,0) rectangle (5,1);
\draw  (5,0) rectangle (6,1);
\draw  (6,0) rectangle (7,1);
\draw  (-1,1) rectangle (0,2);
\draw  (0,1) rectangle (1,2);
\draw  (1,1) rectangle (2,2);
\draw  (2,1) rectangle (3,2);
\draw  (3,1) rectangle (4,2);
\draw  (4,1) rectangle (5,2);
\draw  (5,1) rectangle (6,2);
\draw [ fill={rgb,255:red,0; green,0; blue,0}, fill opacity=1] (-1,0) rectangle (0,1);
\end{circuitikz}
}%
\caption{Representation of an illegal move from $\DCG(f, 7, 4)$ with $f(i)=\max(\lfloor \frac{8-i}{2} \rfloor,0)$}
\label{fig:illegalcutdcg}
\end{figure}

With this additional rule, Chocolate Games can be considered as a generalization of games under the mis\`{e}re play convention, in which the player who moves last loses.
If we play the game under the normal play convention with the additional condition that, as soon as a terminal position is reached, a new one-pile Nim position consisting of a single token appears, then the player who moves to the terminal position loses. Consequently, one must consider the same winning strategy as in mis\`{e}re play. Such an additional condition can be represented naturally in the chocolate game as follows.

Let us consider the Chocolate Game played on the chocolate bar shown in Fig. \ref{fig:mischoco} under the normal play convention. 
Here, one can only reduce the width of this chocolate until only the leftmost column remains. The height can be reduced if and only if only the leftmost column remains. Therefore, the player who reduces the width to zero becomes the loser, which corresponds to one-pile Nim under the mis\`{e}re play convention, where the player who removes the last token loses.

\begin{figure}[!ht]
\centering
\resizebox{1\textwidth/4}{!}{%
\begin{circuitikz}
\tikzstyle{every node}=[font=\fontsize{18.2pt}{23.7pt}\selectfont]
\draw  (0,0) rectangle (1,1);
\draw  (1,0) rectangle (2,1);
\draw  (2,0) rectangle (3,1);
\draw  (3,0) rectangle (4,1);
\draw  (4,0) rectangle (5,1);
\draw  (5,0) rectangle (6,1);
\draw  (6,0) rectangle (7,1);
\draw  (-1,1) rectangle (0,2);
\draw [ fill={rgb,255:red,0; green,0; blue,0}, fill opacity=1] (-1,0) rectangle (0,1);
\end{circuitikz}
}%
\caption{Representation of a chocolate bar on which Chocolate Game corresponds to one-pile Nim under the mis\`{e}re play convention}
\label{fig:mischoco}
\end{figure}
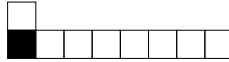

In previous studies \cite{MN21}, a necessary and sufficient condition for the Sprague--Grundy value of an increasing chocolate bar to be equal to the Nim--sum of the coordinates of the position is given.

Figure~\ref{fig:tastychoco} shows, in each block, the Sprague--Grundy value of the position obtained when that block becomes the upper-right block of the remaining chocolate bar, that is, when all blocks above it and to its right have been removed. In this figure, for a block with coordinates $(x,y)$, the Sprague--Grundy value is given by $x\oplus y$, where $\oplus$ denotes the Nim--sum. On the other hand, in the chocolate bar shown in Figure~\ref{fig:nontastychoco}, this property fails for several blocks (described by red numbers). In this paper, we call a chocolate bar {\em tasty} if the Sprague--Grundy value associated with every block is given by the Nim--sum of its coordinates. The term ``tasty'' is a play on the well--known term ``tame,'' which is used for games with well-behaved properties under mis\`{e}re play.

  \begin{figure}[!ht]
\centering
\resizebox{1\textwidth/4}{!}{%
\begin{circuitikz}
\tikzstyle{every node}=[font=\fontsize{18.2pt}{23.7pt}\selectfont]
\draw  (0,0) rectangle (1,1);
\draw  (1,0) rectangle (2,1);
\draw  (2,0) rectangle (3,1);
\draw  (3,0) rectangle (4,1);
\draw  (4,0) rectangle (5,1);
\draw  (5,0) rectangle (6,1);
\draw  (6,0) rectangle (7,1);
\draw  (1,1) rectangle (2,2);
\draw  (2,1) rectangle (3,2);
\draw  (3,1) rectangle (4,2);
\draw  (4,1) rectangle (5,2);
\draw  (5,1) rectangle (6,2);
\draw  (6,1) rectangle (7,2);
\draw  (3,2) rectangle (4,3);
\draw  (4,2) rectangle (5,3);
\draw  (5,2) rectangle (6,3);
\draw  (6,2) rectangle (7,3);
\draw  (5,3) rectangle (6,4);
\draw  (6,3) rectangle (7,4);
\draw [ fill={rgb,255:red,0; green,0; blue,0}, fill opacity=1] (-1,0) rectangle (0,1);
\node [font=\fontsize{18.2pt}{23.7pt}\selectfont, inner xsep=0.080cm, inner ysep=0.085cm, rounded corners=0.020cm, color=white] at (-0.5,0.5) {0};
\node [font=\fontsize{18.2pt}{23.7pt}\selectfont, inner xsep=0.080cm, inner ysep=0.085cm, rounded corners=0.020cm] at (0.5,0.5) {1};
\node [font=\fontsize{18.2pt}{23.7pt}\selectfont, inner xsep=0.080cm, inner ysep=0.085cm, rounded corners=0.020cm] at (1.5,0.5) {2};
\node [font=\fontsize{18.2pt}{23.7pt}\selectfont, inner xsep=0.080cm, inner ysep=0.085cm, rounded corners=0.020cm] at (2.5,0.5) {3};
\node [font=\fontsize{18.2pt}{23.7pt}\selectfont, inner xsep=0.080cm, inner ysep=0.085cm, rounded corners=0.020cm] at (3.5,0.5) {4};
\node [font=\fontsize{18.2pt}{23.7pt}\selectfont, inner xsep=0.080cm, inner ysep=0.085cm, rounded corners=0.020cm] at (4.5,0.5) {5};
\node [font=\fontsize{18.2pt}{23.7pt}\selectfont, inner xsep=0.080cm, inner ysep=0.085cm, rounded corners=0.020cm] at (5.5,0.5) {6};
\node [font=\fontsize{18.2pt}{23.7pt}\selectfont, inner xsep=0.080cm, inner ysep=0.085cm, rounded corners=0.020cm] at (6.5,0.5) {7};
\node [font=\fontsize{18.2pt}{23.7pt}\selectfont, inner xsep=0.080cm, inner ysep=0.085cm, rounded corners=0.020cm] at (1.5,1.5) {3};
\node [font=\fontsize{18.2pt}{23.7pt}\selectfont, inner xsep=0.080cm, inner ysep=0.085cm, rounded corners=0.020cm] at (2.5,1.5) {2};
\node [font=\fontsize{18.2pt}{23.7pt}\selectfont, inner xsep=0.080cm, inner ysep=0.085cm, rounded corners=0.020cm] at (3.5,1.5) {5};
\node [font=\fontsize{18.2pt}{23.7pt}\selectfont, inner xsep=0.080cm, inner ysep=0.085cm, rounded corners=0.020cm] at (4.5,1.5) {4};
\node [font=\fontsize{18.2pt}{23.7pt}\selectfont, inner xsep=0.080cm, inner ysep=0.085cm, rounded corners=0.020cm] at (5.5,1.5) {7};
\node [font=\fontsize{18.2pt}{23.7pt}\selectfont, inner xsep=0.080cm, inner ysep=0.085cm, rounded corners=0.020cm] at (6.5,1.5) {6};
\node [font=\fontsize{18.2pt}{23.7pt}\selectfont, inner xsep=0.080cm, inner ysep=0.085cm, rounded corners=0.020cm] at (3.5,2.5) {6};
\node [font=\fontsize{18.2pt}{23.7pt}\selectfont, inner xsep=0.080cm, inner ysep=0.085cm, rounded corners=0.020cm] at (4.5,2.5) {7};
\node [font=\fontsize{18.2pt}{23.7pt}\selectfont, inner xsep=0.080cm, inner ysep=0.085cm, rounded corners=0.020cm] at (5.5,2.5) {4};
\node [font=\fontsize{18.2pt}{23.7pt}\selectfont, inner xsep=0.080cm, inner ysep=0.085cm, rounded corners=0.020cm] at (6.5,2.5) {5};
\node [font=\fontsize{18.2pt}{23.7pt}\selectfont, inner xsep=0.080cm, inner ysep=0.085cm, rounded corners=0.020cm] at (5.5,3.5) {5};
\node [font=\fontsize{18.2pt}{23.7pt}\selectfont, inner xsep=0.080cm, inner ysep=0.085cm, rounded corners=0.020cm] at (6.5,3.5) {4};
\end{circuitikz}
}%
\caption{Representation of a tasty chocolate bar filled with Sprague--Grundy value}
\label{fig:tastychoco}
\end{figure}

 \begin{figure}[!ht]
\centering
\resizebox{1\textwidth/4}{!}{%
\begin{circuitikz}
\tikzstyle{every node}=[font=\fontsize{18.2pt}{23.7pt}\selectfont]
\draw  (0,0) rectangle (1,1);
\draw  (1,0) rectangle (2,1);
\draw  (2,0) rectangle (3,1);
\draw  (3,0) rectangle (4,1);
\draw  (4,0) rectangle (5,1);
\draw  (5,0) rectangle (6,1);
\draw  (6,0) rectangle (7,1);
\draw  (0,1) rectangle (1,2);
\draw  (1,1) rectangle (2,2);
\draw  (2,1) rectangle (3,2);
\draw  (3,1) rectangle (4,2);
\draw  (4,1) rectangle (5,2);
\draw  (5,1) rectangle (6,2);
\draw  (6,1) rectangle (7,2);
\draw  (3,2) rectangle (4,3);
\draw  (4,2) rectangle (5,3);
\draw  (5,2) rectangle (6,3);
\draw  (6,2) rectangle (7,3);
\draw  (5,3) rectangle (6,4);
\draw  (6,3) rectangle (7,4);
\draw [ fill={rgb,255:red,0; green,0; blue,0}, fill opacity=1] (-1,0) rectangle (0,1);
\node [font=\fontsize{18.2pt}{23.7pt}\selectfont, inner xsep=0.080cm, inner ysep=0.085cm, rounded corners=0.020cm, color=white] at (-0.5,0.5) {0};
\node [font=\fontsize{18.2pt}{23.7pt}\selectfont, inner xsep=0.080cm, inner ysep=0.085cm, rounded corners=0.020cm] at (0.5,0.5) {1};
\node [font=\fontsize{18.2pt}{23.7pt}\selectfont, inner xsep=0.080cm, inner ysep=0.085cm, rounded corners=0.020cm] at (1.5,0.5) {2};
\node [font=\fontsize{18.2pt}{23.7pt}\selectfont, inner xsep=0.080cm, inner ysep=0.085cm, rounded corners=0.020cm] at (2.5,0.5) {3};
\node [font=\fontsize{18.2pt}{23.7pt}\selectfont, inner xsep=0.080cm, inner ysep=0.085cm, rounded corners=0.020cm] at (3.5,0.5) {4};
\node [font=\fontsize{18.2pt}{23.7pt}\selectfont, inner xsep=0.080cm, inner ysep=0.085cm, rounded corners=0.020cm] at (4.5,0.5) {5};
\node [font=\fontsize{18.2pt}{23.7pt}\selectfont, inner xsep=0.080cm, inner ysep=0.085cm, rounded corners=0.020cm] at (5.5,0.5) {6};
\node [font=\fontsize{18.2pt}{23.7pt}\selectfont, inner xsep=0.080cm, inner ysep=0.085cm, rounded corners=0.020cm] at (6.5,0.5) {7};
\node [font=\fontsize{18.2pt}{23.7pt}\selectfont, inner xsep=0.080cm, inner ysep=0.085cm, rounded corners=0.020cm, color=red] at (0.5,1.5) {2};
\node [font=\fontsize{18.2pt}{23.7pt}\selectfont, inner xsep=0.080cm, inner ysep=0.085cm, rounded corners=0.020cm, color=red] at (1.5,1.5) {1};
\node [font=\fontsize{18.2pt}{23.7pt}\selectfont, inner xsep=0.080cm, inner ysep=0.085cm, rounded corners=0.020cm, color=red] at (2.5,1.5) {4};
\node [font=\fontsize{18.2pt}{23.7pt}\selectfont, inner xsep=0.080cm, inner ysep=0.085cm, rounded corners=0.020cm, color=red] at (3.5,1.5) {3};
\node [font=\fontsize{18.2pt}{23.7pt}\selectfont, inner xsep=0.080cm, inner ysep=0.085cm, rounded corners=0.020cm, color=red] at (4.5,1.5) {6};
\node [font=\fontsize{18.2pt}{23.7pt}\selectfont, inner xsep=0.080cm, inner ysep=0.085cm, rounded corners=0.020cm, color=red] at (5.5,1.5) {5};
\node [font=\fontsize{18.2pt}{23.7pt}\selectfont, inner xsep=0.080cm, inner ysep=0.085cm, rounded corners=0.020cm, color=red] at (6.5,1.5) {8};
\node [font=\fontsize{18.2pt}{23.7pt}\selectfont, inner xsep=0.080cm, inner ysep=0.085cm, rounded corners=0.020cm, color=red] at (3.5,2.5) {5};
\node [font=\fontsize{18.2pt}{23.7pt}\selectfont, inner xsep=0.080cm, inner ysep=0.085cm, rounded corners=0.020cm, color=red] at (4.5,2.5) {3};
\node [font=\fontsize{18.2pt}{23.7pt}\selectfont, inner xsep=0.080cm, inner ysep=0.085cm, rounded corners=0.020cm, color=red] at (5.5,2.5) {7};
\node [font=\fontsize{18.2pt}{23.7pt}\selectfont, inner xsep=0.080cm, inner ysep=0.085cm, rounded corners=0.020cm, color=red] at (6.5,2.5) {6};
\node [font=\fontsize{18.2pt}{23.7pt}\selectfont, inner xsep=0.080cm, inner ysep=0.085cm, rounded corners=0.020cm, color=red] at (5.5,3.5) {8};
\node [font=\fontsize{18.2pt}{23.7pt}\selectfont, inner xsep=0.080cm, inner ysep=0.085cm, rounded corners=0.020cm, color=red] at (6.5,3.5) {9};
\end{circuitikz}
}%
\caption{Representation of a non-tasty chocolate bar fill with Sprague--Grundy value}
\label{fig:nontastychoco}
\end{figure}

In the previous studies \cite{MN21}, 
a necessary and sufficient condition for
each of two cases of increasing two-dimensional chocolate bar and increasing three-dimensional chocolate bar to be tasty is shown.
For each case, a lengthy argument requiring an entire paper was necessary. In this study, our result gives a simpler proof that subsumes these previous results. We give a necessary and sufficient condition to be tasty not only for the cases in previous studies, but also increasing chocolate bars in two and higher dimensional cases. We also show a sufficient condition to be tasty for more complex shaped chocolate bars. 
This is one of the main results of this paper, and we believe that it makes a significant contribution to the study of Chocolate Games.

Furthermore, from these results, we show relationships between these chocolate bars.

Generalizations of mis\`{e}re play and games with a pass have already been studied, as in \cite{S26}. As we have seen above, mis\`{e}re play can be regarded as a game in which a one-pile Nim component with one token is created as soon as a terminal position is reached, whereas a game with a pass can be regarded as one in which such a component disappears at the same moment. The generalizations considered in the paper replace this terminal-position condition by allowing a one-pile Nim component with one token to be created or removed at a specified point during the play of the game.

In such variants, the Sprague--Grundy value of the entire position in the generalization of misère play is equal to the Nim--sum of the Sprague--Grundy value of the original position and the size of the additional Nim pile if and only if the analogous equality holds for the corresponding generalization of games with a pass.
This can be applied for Chocolate Games. This result shows a correspondence that both chocolate bars in Figs. \ref{fig:tastypass} and \ref{fig:tastymisere} are tasty; where the steps of one staircase chocolate bar become lower, those of the other staircase chocolate bar become higher. 
In this paper, we show a further generalized result:
If an increasing chocolate bar is tasty, there exists a corresponding decreasing tasty chocolate bar, and the cross-sections of the two pieces of chocolate fit together perfectly to form a single rectangle, as shown in Fig. \ref{fig:inc_dec_tasty}.

\begin{figure}[!ht]
\centering
\resizebox{1\textwidth/4}{!}{%
\begin{circuitikz}
\tikzstyle{every node}=[font=\fontsize{18.2pt}{23.7pt}\selectfont]
\draw  (0,0) rectangle (1,1);
\draw  (1,0) rectangle (2,1);
\draw  (2,0) rectangle (3,1);
\draw  (3,0) rectangle (4,1);
\draw  (4,0) rectangle (5,1);
\draw  (5,0) rectangle (6,1);
\draw  (6,0) rectangle (7,1);
\draw  (3,1) rectangle (4,2);
\draw  (4,1) rectangle (5,2);
\draw  (5,1) rectangle (6,2);
\draw  (6,1) rectangle (7,2);
\draw [ fill={rgb,255:red,0; green,0; blue,0}, fill opacity=1] (-1,0) rectangle (0,1);
\node [font=\fontsize{18.2pt}{23.7pt}\selectfont, inner xsep=0.080cm, inner ysep=0.085cm, rounded corners=0.020cm, color=white] at (-0.5,0.5) {0};
\node [font=\fontsize{18.2pt}{23.7pt}\selectfont, inner xsep=0.080cm, inner ysep=0.085cm, rounded corners=0.020cm] at (0.5,0.5) {1};
\node [font=\fontsize{18.2pt}{23.7pt}\selectfont, inner xsep=0.080cm, inner ysep=0.085cm, rounded corners=0.020cm] at (1.5,0.5) {2};
\node [font=\fontsize{18.2pt}{23.7pt}\selectfont, inner xsep=0.080cm, inner ysep=0.085cm, rounded corners=0.020cm] at (2.5,0.5) {3};
\node [font=\fontsize{18.2pt}{23.7pt}\selectfont, inner xsep=0.080cm, inner ysep=0.085cm, rounded corners=0.020cm] at (3.5,0.5) {4};
\node [font=\fontsize{18.2pt}{23.7pt}\selectfont, inner xsep=0.080cm, inner ysep=0.085cm, rounded corners=0.020cm] at (4.5,0.5) {5};
\node [font=\fontsize{18.2pt}{23.7pt}\selectfont, inner xsep=0.080cm, inner ysep=0.085cm, rounded corners=0.020cm] at (5.5,0.5) {6};
\node [font=\fontsize{18.2pt}{23.7pt}\selectfont, inner xsep=0.080cm, inner ysep=0.085cm, rounded corners=0.020cm] at (6.5,0.5) {7};
\node [font=\fontsize{18.2pt}{23.7pt}\selectfont, inner xsep=0.080cm, inner ysep=0.085cm, rounded corners=0.020cm] at (3.5,1.5) {5};
\node [font=\fontsize{18.2pt}{23.7pt}\selectfont, inner xsep=0.080cm, inner ysep=0.085cm, rounded corners=0.020cm] at (4.5,1.5) {4};
\node [font=\fontsize{18.2pt}{23.7pt}\selectfont, inner xsep=0.080cm, inner ysep=0.085cm, rounded corners=0.020cm] at (5.5,1.5) {7};
\node [font=\fontsize{18.2pt}{23.7pt}\selectfont, inner xsep=0.080cm, inner ysep=0.085cm, rounded corners=0.020cm] at (6.5,1.5) {6};
\end{circuitikz}
}%
\caption{Representation of a tasty one staircase chocolate bar filled with Sprague--Grundy value}
\label{fig:tastypass}
\end{figure}

\begin{figure}[!ht]
\centering
\resizebox{1\textwidth/4}{!}{%
\begin{circuitikz}
\tikzstyle{every node}=[font=\fontsize{18.2pt}{23.7pt}\selectfont]
\draw  (0,0) rectangle (1,1);
\draw  (1,0) rectangle (2,1);
\draw  (2,0) rectangle (3,1);
\draw  (3,0) rectangle (4,1);
\draw  (4,0) rectangle (5,1);
\draw  (5,0) rectangle (6,1);
\draw  (6,0) rectangle (7,1);
\draw  (-1,1) rectangle (0,2);
\draw  (0,1) rectangle (1,2);
\draw  (1,1) rectangle (2,2);
\draw  (2,1) rectangle (3,2);
\draw [ fill={rgb,255:red,0; green,0; blue,0}, fill opacity=1] (-1,0) rectangle (0,1);
\node [font=\fontsize{18.2pt}{23.7pt}\selectfont, inner xsep=0.080cm, inner ysep=0.085cm, rounded corners=0.020cm, color=white] at (-0.5,0.5) {0};
\node [font=\fontsize{18.2pt}{23.7pt}\selectfont, inner xsep=0.080cm, inner ysep=0.085cm, rounded corners=0.020cm] at (0.5,0.5) {1};
\node [font=\fontsize{18.2pt}{23.7pt}\selectfont, inner xsep=0.080cm, inner ysep=0.085cm, rounded corners=0.020cm] at (1.5,0.5) {2};
\node [font=\fontsize{18.2pt}{23.7pt}\selectfont, inner xsep=0.080cm, inner ysep=0.085cm, rounded corners=0.020cm] at (2.5,0.5) {3};
\node [font=\fontsize{18.2pt}{23.7pt}\selectfont, inner xsep=0.080cm, inner ysep=0.085cm, rounded corners=0.020cm] at (3.5,0.5) {4};
\node [font=\fontsize{18.2pt}{23.7pt}\selectfont, inner xsep=0.080cm, inner ysep=0.085cm, rounded corners=0.020cm] at (4.5,0.5) {5};
\node [font=\fontsize{18.2pt}{23.7pt}\selectfont, inner xsep=0.080cm, inner ysep=0.085cm, rounded corners=0.020cm] at (5.5,0.5) {6};
\node [font=\fontsize{18.2pt}{23.7pt}\selectfont, inner xsep=0.080cm, inner ysep=0.085cm, rounded corners=0.020cm] at (6.5,0.5) {7};
\node [font=\fontsize{18.2pt}{23.7pt}\selectfont, inner xsep=0.080cm, inner ysep=0.085cm, rounded corners=0.020cm] at (-0.5,1.5) {1};
\node [font=\fontsize{18.2pt}{23.7pt}\selectfont, inner xsep=0.080cm, inner ysep=0.085cm, rounded corners=0.020cm] at (0.5,1.5) {0};
\node [font=\fontsize{18.2pt}{23.7pt}\selectfont, inner xsep=0.080cm, inner ysep=0.085cm, rounded corners=0.020cm] at (1.5,1.5) {3};
\node [font=\fontsize{18.2pt}{23.7pt}\selectfont, inner xsep=0.080cm, inner ysep=0.085cm, rounded corners=0.020cm] at (2.5,1.5) {2};
\end{circuitikz}
}%
\caption{Representation of a tasty one staircase chocolate bar filled with Sprague--Grundy value}
\label{fig:tastymisere}
\end{figure}

\begin{figure}[!ht]
\centering
\resizebox{1\textwidth/4}{!}{%
\begin{circuitikz}
\tikzstyle{every node}=[font=\fontsize{18.2pt}{23.7pt}\selectfont]
\draw [color=white]  (-2,-1) rectangle (8,7);
\draw  (0,0) rectangle (1,1);
\draw  (1,0) rectangle (2,1);
\draw  (2,0) rectangle (3,1);
\draw  (3,0) rectangle (4,1);
\draw  (4,0) rectangle (5,1);
\draw  (5,0) rectangle (6,1);
\draw  (6,0) rectangle (7,1);
\draw  (1,1) rectangle (2,2);
\draw  (2,1) rectangle (3,2);
\draw  (3,1) rectangle (4,2);
\draw  (4,1) rectangle (5,2);
\draw  (5,1) rectangle (6,2);
\draw  (6,1) rectangle (7,2);
\draw  (3,2) rectangle (4,3);
\draw  (4,2) rectangle (5,3);
\draw  (5,2) rectangle (6,3);
\draw  (6,2) rectangle (7,3);
\draw  (5,3) rectangle (6,4);
\draw  (6,3) rectangle (7,4);
\draw [ fill={rgb,255:red,0; green,0; blue,0}, fill opacity=1] (-1,0) rectangle (0,1);

\draw  (-1,5) rectangle (0,6);
\draw  (-1,4) rectangle (0,5);
\draw  (-1,3) rectangle (0,4);
\draw  (0,6) rectangle (1,7);
\draw  (0,5) rectangle (1,6);
\draw  (0,4) rectangle (1,5);
\draw  (0,3) rectangle (1,4);
\draw  (1,6) rectangle (2,7);
\draw  (1,5) rectangle (2,6);
\draw  (1,4) rectangle (2,5);
\draw  (2,6) rectangle (3,7);
\draw  (2,5) rectangle (3,6);
\draw  (2,4) rectangle (3,5);
\draw  (3,6) rectangle (4,7);
\draw  (3,5) rectangle (4,6);
\draw  (4,6) rectangle (5,7);
\draw  (4,5) rectangle (5,6);
\draw  (5,6) rectangle (6,7);
\draw  (6,6) rectangle (7,7);
\draw [ fill={rgb,255:red,0; green,0; blue,0}, fill opacity=1] (-1,6) rectangle (0,7);
\end{circuitikz}
}%
\caption{Representation of one increasing tasty chocolate bar and a corresponding decreasing tasty chocolate bar forming a single rectangle}
\label{fig:inc_dec_tasty}
\end{figure}

The remainder of this paper is organized as follows. In the rest of this section, we introduce the necessary definitions and review previous work. 
In the next section, we present several results on functions satisfying a condition called tasty condition,  which plays an important role in this paper.
In Section \ref{sec:main}, we first present our results for the two-dimensional Chocolate Games, and later, we show that analogous results can also be obtained in higher dimensions.

\subsection{Preliminaries}

In this paper, we only consider {\em impartial games}. That is, two-player games such that in every position, the sets of options for the both players are the same.

For impartial games, Sprague--Grundy value is an important parameter.


\begin{definition}
For any position $G$ in an impartial game, its {\em Sprague--Grundy value} $\mathcal{G}(G)$ is    
$$\mathcal{G}(G) = {\rm mex}(\{\mathcal{G}(G')\mid G' \text{ is an option of } G\}),$$
where ${\rm mex}(S) = \min(\mathbb{Z}_{\ge 0}\setminus S)$.
\end{definition}

Sprague--Grundy values are independently introduced by Sprague \cite{S35} and Grundy \cite{G39}. 
The reason why Sprague--Grundy values are important is that the value is zero is a necessary and sufficient condition  for the previous player has a winning strategy in the position. 

Furthermore, the Sprague--Grundy values play a crucial role in the analysis of combinations of games known as disjunctive sums.
A disjunctive sum of games is a game such that there are some game positions and each player chooses exactly one position to make a move, and the turn ends. The play ends when every component becomes a terminal position. The Sprague--Grundy value of a position which can be considered as a disjunctive sum of some games is known to be equal to bitwise-XOR of binary notations (or Nim--sum) of Sprague--Grundy values of the components. The Nim--sum is denoted as $\oplus$.

Note that for $n$-pile Nim whose numbers of tokens are $a_1, \ldots, a_n,$ its Sprague--Grundy values is $a_1 \oplus \cdots \oplus a_n$ since the game is a disjunctive sum of one-pile Nim positions and it is easy to confirm that the Sprague--Grundy value of a one-pile Nim position is the same as its number of tokens.

In this paper, we consider the Sprague--Grundy values of Chocolate Games.
Since, as shown below, Chocolate Game on rectangular chocolate bar is isomorphic to a two--pile Nim and its Sprague--Grundy value can be obtained by the Nim--sum of the width and height of the rectangular chocolate bar, previous studies have investigated the conditions under which the Sprague--Grundy value is given by the Nim--sum of the height and width for non-rectangular chocolate bars.

\begin{definition}[Rectangular Chocolate Game]
     For nonnegative integers $x$ and $y$, we consider the chocolate bar  consists of $x+1$ columns where the bottom of $0$-th column is the bitter square, and the height of the $i$-th column is $t(i) =  y+1$ for $i = 0,1,\ldots ,x$. We denote this by $\SCG(x, y)$.
\end{definition}

\begin{theorem}
   For all $x,y$ in $\mathbb{Z}_{\ge 0}$, $g(\SCG(x,y))=x\oplus y$.
\end{theorem}

\begin{proof}
    Rectangular Chocolate Game is isomorphic to two-pile Nim as we can consider the two dimensions as two different piles. Thus, for all $x,y$, $g(\SCG(x,y))=x\oplus y$. 
\end{proof}

We only focus on chocolate bars where the bitter square is in the lower left corner and the player have to take one of the rightmost square at each turn.


In \cite{MN21}, the authors characterize when a chocolate bar whose height increases from right to left has the same Sprague--Grundy value as the rectangular chocolate bar with the same height and width.

\begin{definition}
\label{def:func_increase}
    Let $f$ be a function that satisfies the following two conditions:
    \begin{itemize}
        \item $f(t)\in \mathbb{Z}_{\ge 0}$ for $t \in \mathbb{Z}_{\ge 0}$. 
        \item $f$ is monotonically increasing, i.e., we have $f(u)\le f(v)$ for $u,v\in \mathbb{Z}_{\ge 0}$ with $u\le v$.
    \end{itemize}
    Then, $f$ is an {\em increasing chocolate function}.
\end{definition}

\begin{definition}[Increasing Chocolate Game]
    Let $f$ be an increasing chocolate function. For nonnegative integers $x$, $y$, we consider the chocolate bar consists of $x+1$ columns where the bottom of $0$-th column is the bitter square, and the height of the $i$-th column is $t(i) = \min(f (i), y) + 1$ for $i = 0,1,\ldots  ,x$. We  denote this by $\CB(f, x, y)$.
\end{definition}

\begin{definition}[Tasty condition]
\label{def:conditionCB}
    Let $h$ be a function of $\mathbb{Z}_{\ge 0}^n$ into $\mathbb{Z}_{\ge 0}$.
    $h$ satisfies {\em tasty condition} if: 

Assume that 
    \begin{equation}
        \left\lfloor \frac{z}{2^i} \right\rfloor=\left\lfloor \frac{z'}{2^i} \right\rfloor
    \end{equation}
    for some $z,z' \in \mathbb{Z}_{\ge 0} $ and some positive integer $i$. Then we have for all $x_1,\ldots  ,x_n \in \mathbb{Z}_{\ge 0}$
    \begin{equation}
        \left \lfloor \frac{h(x_1,x_2,\ldots  x_{k-1},z,x_{k+1},\ldots  ,x_n)}{2^{i-1}} \right \rfloor= \left \lfloor \frac{h(x_1,x_2,\ldots  x_{k-1},z',x_{k+1},\ldots  ,x_n)}{2^{i-1}} \right\rfloor
    \end{equation}
\end{definition}

The following Theorem \ref{thm:inctasty} is shown in  \cite{MNN20}.

\begin{definition}
    Let $f$ be a function from $\mathbb{Z}_{\ge 0}$ to $\mathbb{Z}_{\ge 0}.$
    $\CB(f,x,y)$ is {\em tasty} if for every $0\le i \le x$ and $0 \le j \le y$, $$g(\CB(f,i,j))  = i\oplus \min(f(i),j).$$
    Similarly, $\DCG(f,x,y)$ is {\em tasty} if for every $0\le i \le x$ and $0 \le j \le y$, $$g(\DCG(f,i,j))  = i\oplus \min(f(i),j).$$
\end{definition}

\begin{theorem}
\label{thm:inctasty}
    $\CB(f, x, y)$ is tasty for any nonnegative integers $x, y$ if and only if $f$ satisfies the tasty condition. 
\end{theorem}

Similarly to the Chocolate Game on increasing chocolate bar, we characterize the case when a chocolate bar whose height decreases from right to left has the same Sprague--Grundy value as the rectangular  chocolate bar with the same height and width.

\begin{definition}
\label{def:func_decrease}
    Let $f$ be a function that satisfies the following two conditions:
    \begin{itemize}
        \item $f(t)\in \mathbb{Z}_{\ge 0}$ for $t \in \mathbb{Z}_{\ge 0}$.
        \item $f$ is monotonically decreasing, i.e., we have $f(u)\ge f(v)$ for $u,v\in \mathbb{Z}_{\ge 0}$ with $u\le v$.
    \end{itemize}

    Then, $f$ is a {\em decreasing chocolate function}.
\end{definition}

\begin{definition}[Decreasing Chocolate Game]
    Let $f$ be a decreasing chocolate function. For nonnegative integers $x$ and $y$, We consider the chocolate bar consists of $x+1$ columns where the bottom of $0$-th column is the bitter square, and the height of the $i$-th column is $t(i) = \min(f (i), y) + 1$ for $i = 0,1,\ldots  ,x$. We add the rule that at least one square on the rightmost column must be removed on each turn. We denote this by $\DCG(f, x, y)$.
\end{definition}

\begin{remark}
    From the definition, the set of options of $\DCG(f,x,y)$ is $$\{\DCG(f,x',y),\DCG(f,x,y')\mid0\le x'<x,0\le y'<\min(f(x),y)\}.$$
\end{remark}


We also consider chocolate bars with more complex shapes. 
Changing Height Chocolate Game allow not only increasing and decreasing function but any function allowing more type of chocolate bars as represented in Figure~\ref{fig:tasty_chn}.

\begin{definition}[Changing Height Chocolate Game]
    Let $f$ be a function from $\mathbb{Z}_{\ge 0}$ into $\mathbb{Z}_{\ge 0}$. For $x$, $y$, we consider the chocolate bar consists of $x+1$ columns where the $0$-th column is the bitter square, and the height of the $i$-th column is $t(i) = \min(f (i), y) + 1$ for $i = 0,1,\ldots  ,x$. We add the rule that at least one square on the rightmost column must be removed on  each turn. We denote this by $\CHCG(f, x, y)$.
\end{definition}

Figure \ref{fig:tasty_chn} shows a tasty position in Changing Height Chocolate Game.  
Note that the definition of tasty positions that are not in Increasing Chocolate Games or Decreasing Chocolate Games is given later as Definition \ref{def:tastyhigher}.  

\begin{figure}[!ht]
\centering
\resizebox{1\textwidth/2}{!}{%
\begin{circuitikz}
\tikzstyle{every node}=[font=\fontsize{18.2pt}{23.7pt}\selectfont]
\draw [color=white]  (-2,-1) rectangle (8,7);
\draw  (0,0) rectangle (1,1);
\draw  (1,0) rectangle (2,1);
\draw  (2,0) rectangle (3,1);
\draw  (3,0) rectangle (4,1);
\draw  (4,0) rectangle (5,1);
\draw  (5,0) rectangle (6,1);
\draw  (6,0) rectangle (7,1);
\draw  (-1,1) rectangle (0,2);
\draw  (0,1) rectangle (1,2);
\draw  (3,1) rectangle (4,2);
\draw  (4,1) rectangle (5,2);
\draw  (5,1) rectangle (6,2);
\draw  (6,1) rectangle (7,2);
\draw  (3,2) rectangle (4,3);
\draw  (4,2) rectangle (5,3);
\draw  (5,2) rectangle (6,3);
\draw  (6,2) rectangle (7,3);
\draw  (3,3) rectangle (4,4);
\draw  (4,3) rectangle (5,4);

\draw [ fill={rgb,255:red,0; green,0; blue,0}, fill opacity=1] (-1,0) rectangle (0,1);
\end{circuitikz}

\begin{circuitikz}
\tikzstyle{every node}=[font=\fontsize{18.2pt}{23.7pt}\selectfont]
\draw [color=white]  (-2,-1) rectangle (8,7);
\draw  (0,0) rectangle (1,1);
\draw  (1,0) rectangle (2,1);
\draw  (2,0) rectangle (3,1);
\draw  (3,0) rectangle (4,1);
\draw  (4,0) rectangle (5,1);
\draw  (5,0) rectangle (6,1);
\draw  (6,0) rectangle (7,1);
\draw  (-1,1) rectangle (0,2);
\draw  (0,1) rectangle (1,2);
\draw  (3,1) rectangle (4,2);
\draw  (4,1) rectangle (5,2);
\draw  (5,1) rectangle (6,2);
\draw  (6,1) rectangle (7,2);
\draw  (3,2) rectangle (4,3);
\draw  (4,2) rectangle (5,3);
\draw  (5,2) rectangle (6,3);
\draw  (6,2) rectangle (7,3);
\draw  (3,3) rectangle (4,4);
\draw  (4,3) rectangle (5,4);
\draw [ fill={rgb,255:red,0; green,0; blue,0}, fill opacity=1] (-1,0) rectangle (0,1);
\node [font=\fontsize{18.2pt}{23.7pt}\selectfont, inner xsep=0.080cm, inner ysep=0.085cm, rounded corners=0.020cm, color=white] at (-0.5,0.5) {0};
\node [font=\fontsize{18.2pt}{23.7pt}\selectfont, inner xsep=0.080cm, inner ysep=0.085cm, rounded corners=0.020cm] at (0.5,0.5) {1};
\node [font=\fontsize{18.2pt}{23.7pt}\selectfont, inner xsep=0.080cm, inner ysep=0.085cm, rounded corners=0.020cm] at (1.5,0.5) {2};
\node [font=\fontsize{18.2pt}{23.7pt}\selectfont, inner xsep=0.080cm, inner ysep=0.085cm, rounded corners=0.020cm] at (2.5,0.5) {3};
\node [font=\fontsize{18.2pt}{23.7pt}\selectfont, inner xsep=0.080cm, inner ysep=0.085cm, rounded corners=0.020cm] at (3.5,0.5) {4};
\node [font=\fontsize{18.2pt}{23.7pt}\selectfont, inner xsep=0.080cm, inner ysep=0.085cm, rounded corners=0.020cm] at (4.5,0.5) {5};
\node [font=\fontsize{18.2pt}{23.7pt}\selectfont, inner xsep=0.080cm, inner ysep=0.085cm, rounded corners=0.020cm] at (5.5,0.5) {6};
\node [font=\fontsize{18.2pt}{23.7pt}\selectfont, inner xsep=0.080cm, inner ysep=0.085cm, rounded corners=0.020cm] at (6.5,0.5) {7};
\node [font=\fontsize{18.2pt}{23.7pt}\selectfont, inner xsep=0.080cm, inner ysep=0.085cm, rounded corners=0.020cm] at (-0.5,1.5) {1};
\node [font=\fontsize{18.2pt}{23.7pt}\selectfont, inner xsep=0.080cm, inner ysep=0.085cm, rounded corners=0.020cm] at (0.5,1.5) {0};
\node [font=\fontsize{18.2pt}{23.7pt}\selectfont, inner xsep=0.080cm, inner ysep=0.085cm, rounded corners=0.020cm] at (3.5,1.5) {5};
\node [font=\fontsize{18.2pt}{23.7pt}\selectfont, inner xsep=0.080cm, inner ysep=0.085cm, rounded corners=0.020cm] at (4.5,1.5) {4};
\node [font=\fontsize{18.2pt}{23.7pt}\selectfont, inner xsep=0.080cm, inner ysep=0.085cm, rounded corners=0.020cm] at (5.5,1.5) {7};
\node [font=\fontsize{18.2pt}{23.7pt}\selectfont, inner xsep=0.080cm, inner ysep=0.085cm, rounded corners=0.020cm] at (6.5,1.5) {6};
\node [font=\fontsize{18.2pt}{23.7pt}\selectfont, inner xsep=0.080cm, inner ysep=0.085cm, rounded corners=0.020cm] at (3.5,2.5) {6};
\node [font=\fontsize{18.2pt}{23.7pt}\selectfont, inner xsep=0.080cm, inner ysep=0.085cm, rounded corners=0.020cm] at (4.5,2.5) {7};
\node [font=\fontsize{18.2pt}{23.7pt}\selectfont, inner xsep=0.080cm, inner ysep=0.085cm, rounded corners=0.020cm] at (5.5,2.5) {4};
\node [font=\fontsize{18.2pt}{23.7pt}\selectfont, inner xsep=0.080cm, inner ysep=0.085cm, rounded corners=0.020cm] at (6.5,2.5) {5};
\node [font=\fontsize{18.2pt}{23.7pt}\selectfont, inner xsep=0.080cm, inner ysep=0.085cm, rounded corners=0.020cm] at (3.5,3.5) {7};
\node [font=\fontsize{18.2pt}{23.7pt}\selectfont, inner xsep=0.080cm, inner ysep=0.085cm, rounded corners=0.020cm] at (4.5,3.5) {6};

\end{circuitikz}
}%
\caption{Representation of the tasty chocolate bar $\CHN(f, 7, 3)$ with $f(i)=1+4\lfloor \frac{i}{4} \rfloor-\lfloor\frac{i}{2}\rfloor$ }
\label{fig:tasty_chn}
\end{figure}
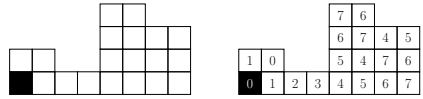

As it is difficult to represent a chocolate bar in more than $3$ dimensions, in order to have a better representation of Chocolate Games in higher dimension, we use a representation starting from the game of Nim since Nim is equivalent to Rectangular Chocolate Game. Then the following Changing Head Nim and Multiple Changing Head Nim act as generalizations of Changing Height Chocolate Game in higher dimension. 

\begin{definition}[Nim]
    For $x_1,\ldots  ,x_k$, the Nim game consists of $k$ piles of $x_i$ for $1\le i\le k$ is denoted by $\MRCG(x_1,\ldots  ,x_k)$.
\end{definition}

\begin{definition}[Changing Head Nim]
    Let $f$ be a function from $\mathbb{Z}_{\ge 0}^k$ to $\mathbb{Z}_{\ge 0}$ for some nonnegative integer $k$. For $y$, $x_1,\ldots  ,x_k$, the Changing Head Nim consists of $k$ piles of $x_i$ for $0\le i\le k$ and the last pile whose size is $\min(f(x_1,\ldots  ,x_k),y)$ which can be changed when we play on another pile. We denote this by $\CHN(f,x_1,\ldots  ,x_k,  y)$.
\end{definition}

\begin{remark}
    Changing Height Chocolate Game is isomorphic to Changing Head Nim when $k=1$.
\end{remark}

By using the framework of Changing Head Nim, we can consider three dimensional Chocolate Games like shown in Fig. \ref{fig:3d_CHN}.

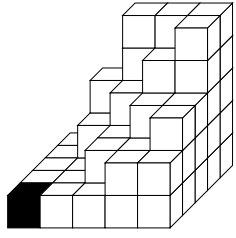
\begin{figure}[!ht]
    \centering
\resizebox{1\textwidth/4}{!}{%
\begin{tikzpicture}
	\draw [ fill={rgb,255:red,255; green,255; blue,255}, fill opacity=1] (0,1,0) -- (1,1,0) -- (1,1,1) -- (0,1,1) -- cycle;

	\draw [ fill={rgb,255:red,255; green,255; blue,255}, fill opacity=1] (1,1,0) -- (2,1,0) -- (2,1,1) -- (1,1,1) -- cycle;

	\draw [ fill={rgb,255:red,255; green,255; blue,255}, fill opacity=1] (2,1,0) -- (3,1,0) -- (3,1,1) -- (2,1,1) -- cycle;

	\draw[ fill={rgb,255:red,255; green,255; blue,255}, fill opacity=1] (3,1,0) -- (4,1,0) -- (4,1,1) -- (3,1,1) -- cycle;

    \draw[ fill={rgb,255:red,255; green,255; blue,255}, fill opacity=1] (5,0,0) -- (5,1,0) -- (5,1,1) -- (5,0,1) -- cycle;
	\draw[ fill={rgb,255:red,255; green,255; blue,255}, fill opacity=1] (4,1,0) -- (5,1,0) -- (5,1,1) -- (4,1,1) -- cycle;

	\draw[ fill={rgb,255:red,255; green,255; blue,255}, fill opacity=1] (0,1,1) -- (1,1,1) -- (1,1,2) -- (0,1,2) -- cycle;

	\draw[ fill={rgb,255:red,255; green,255; blue,255}, fill opacity=1] (1,1,1) -- (2,1,1) -- (2,1,2) -- (1,1,2) -- cycle;

	\draw[ fill={rgb,255:red,255; green,255; blue,255}, fill opacity=1] (2,1,1) -- (3,1,1) -- (3,1,2) -- (2,1,2) -- cycle;

	\draw[ fill={rgb,255:red,255; green,255; blue,255}, fill opacity=1] (3,1,1) -- (4,1,1) -- (4,1,2) -- (3,1,2) -- cycle;

    \draw[ fill={rgb,255:red,255; green,255; blue,255}, fill opacity=1] (5,0,1) -- (5,1,1) -- (5,1,2) -- (5,0,2) -- cycle;
	\draw[ fill={rgb,255:red,255; green,255; blue,255}, fill opacity=1] (4,1,1) -- (5,1,1) -- (5,1,2) -- (4,1,2) -- cycle;

	\draw[ fill={rgb,255:red,255; green,255; blue,255}, fill opacity=1] (0,1,2) -- (1,1,2) -- (1,1,3) -- (0,1,3) -- cycle;

	\draw[ fill={rgb,255:red,255; green,255; blue,255}, fill opacity=1] (1,1,2) -- (2,1,2) -- (2,1,3) -- (1,1,3) -- cycle;

	\draw[ fill={rgb,255:red,255; green,255; blue,255}, fill opacity=1] (2,1,2) -- (3,1,2) -- (3,1,3) -- (2,1,3) -- cycle;

	\draw[ fill={rgb,255:red,255; green,255; blue,255}, fill opacity=1] (3,1,2) -- (4,1,2) -- (4,1,3) -- (3,1,3) -- cycle;

    \draw[ fill={rgb,255:red,255; green,255; blue,255}, fill opacity=1] (5,0,2) -- (5,1,2) -- (5,1,3) -- (5,0,3) -- cycle;
	\draw[ fill={rgb,255:red,255; green,255; blue,255}, fill opacity=1] (4,1,2) -- (5,1,2) -- (5,1,3) -- (4,1,3) -- cycle;

	\draw[ fill={rgb,255:red,255; green,255; blue,255}, fill opacity=1] (0,1,3) -- (1,1,3) -- (1,1,4) -- (0,1,4) -- cycle;

	\draw[ fill={rgb,255:red,255; green,255; blue,255}, fill opacity=1] (1,1,3) -- (2,1,3) -- (2,1,4) -- (1,1,4) -- cycle;

	\draw[ fill={rgb,255:red,255; green,255; blue,255}, fill opacity=1] (2,1,3) -- (3,1,3) -- (3,1,4) -- (2,1,4) -- cycle;

	\draw[ fill={rgb,255:red,255; green,255; blue,255}, fill opacity=1] (3,1,3) -- (4,1,3) -- (4,1,4) -- (3,1,4) -- cycle;

    \draw[ fill={rgb,255:red,255; green,255; blue,255}, fill opacity=1] (5,0,3) -- (5,1,3) -- (5,1,4) -- (5,0,4) -- cycle;
	\draw[ fill={rgb,255:red,255; green,255; blue,255}, fill opacity=1] (4,1,3) -- (5,1,3) -- (5,1,4) -- (4,1,4) -- cycle;

	\draw[ fill={rgb,255:red,0; green,0; blue,0}, fill opacity=1] (0,1,4) -- (1,1,4) -- (1,1,5) -- (0,1,5) -- cycle;
	\draw[ fill={rgb,255:red,0; green,0; blue,0}, fill opacity=1] (0,0,5) -- (0,1,5) -- (1,1,5) -- (1,0,5) -- cycle;

	\draw[ fill={rgb,255:red,255; green,255; blue,255}, fill opacity=1] (1,1,4) -- (2,1,4) -- (2,1,5) -- (1,1,5) -- cycle;
	\draw[ fill={rgb,255:red,255; green,255; blue,255}, fill opacity=1] (1,0,5) -- (1,1,5) -- (2,1,5) -- (2,0,5) -- cycle;

	\draw[ fill={rgb,255:red,255; green,255; blue,255}, fill opacity=1] (2,1,4) -- (3,1,4) -- (3,1,5) -- (2,1,5) -- cycle;
	\draw[ fill={rgb,255:red,255; green,255; blue,255}, fill opacity=1] (2,0,5) -- (2,1,5) -- (3,1,5) -- (3,0,5) -- cycle;

	\draw[ fill={rgb,255:red,255; green,255; blue,255}, fill opacity=1] (3,1,4) -- (4,1,4) -- (4,1,5) -- (3,1,5) -- cycle;
	\draw[ fill={rgb,255:red,255; green,255; blue,255}, fill opacity=1] (3,0,5) -- (3,1,5) -- (4,1,5) -- (4,0,5) -- cycle;

    \draw[ fill={rgb,255:red,255; green,255; blue,255}, fill opacity=1] (5,0,4) -- (5,1,4) -- (5,1,5) -- (5,0,5) -- cycle;
	\draw[ fill={rgb,255:red,255; green,255; blue,255}, fill opacity=1] (4,1,4) -- (5,1,4) -- (5,1,5) -- (4,1,5) -- cycle;
	\draw[ fill={rgb,255:red,255; green,255; blue,255}, fill opacity=1] (4,0,5) -- (4,1,5) -- (5,1,5) -- (5,0,5) -- cycle;

	\draw[ fill={rgb,255:red,255; green,255; blue,255}, fill opacity=1](1,2,0) -- (2,2,0) -- (2,2,1) -- (1,2,1) -- cycle;

	\draw [ fill={rgb,255:red,255; green,255; blue,255}, fill opacity=1] (2,2,0) -- (3,2,0) -- (3,2,1) -- (2,2,1) -- cycle;

	\draw[ fill={rgb,255:red,255; green,255; blue,255}, fill opacity=1] (3,2,0) -- (4,2,0) -- (4,2,1) -- (3,2,1) -- cycle;

    \draw[ fill={rgb,255:red,255; green,255; blue,255}, fill opacity=1] (5,1,0) -- (5,2,0) -- (5,2,1) -- (5,1,1) -- cycle;
	\draw[ fill={rgb,255:red,255; green,255; blue,255}, fill opacity=1] (4,2,0) -- (5,2,0) -- (5,2,1) -- (4,2,1) -- cycle;

	\draw[ fill={rgb,255:red,255; green,255; blue,255}, fill opacity=1] (1,2,1) -- (2,2,1) -- (2,2,2) -- (1,2,2) -- cycle;
	\draw[ fill={rgb,255:red,255; green,255; blue,255}, fill opacity=1] (1,1,2) -- (1,2,2) -- (2,2,2) -- (2,1,2) -- cycle;

	\draw[ fill={rgb,255:red,255; green,255; blue,255}, fill opacity=1] (2,2,1) -- (3,2,1) -- (3,2,2) -- (2,2,2) -- cycle;

	\draw[ fill={rgb,255:red,255; green,255; blue,255}, fill opacity=1] (3,2,1) -- (4,2,1) -- (4,2,2) -- (3,2,2) -- cycle;

    \draw[ fill={rgb,255:red,255; green,255; blue,255}, fill opacity=1] (5,1,1) -- (5,2,1) -- (5,2,2) -- (5,1,2) -- cycle;
	\draw[ fill={rgb,255:red,255; green,255; blue,255}, fill opacity=1] (4,2,1) -- (5,2,1) -- (5,2,2) -- (4,2,2) -- cycle;

	\draw[ fill={rgb,255:red,255; green,255; blue,255}, fill opacity=1] (2,2,2) -- (3,2,2) -- (3,2,3) -- (2,2,3) -- cycle;

	\draw[ fill={rgb,255:red,255; green,255; blue,255}, fill opacity=1] (3,2,2) -- (4,2,2) -- (4,2,3) -- (3,2,3) -- cycle;

    \draw[ fill={rgb,255:red,255; green,255; blue,255}, fill opacity=1] (5,1,2) -- (5,2,2) -- (5,2,3) -- (5,1,3) -- cycle;
	\draw[ fill={rgb,255:red,255; green,255; blue,255}, fill opacity=1] (4,2,2) -- (5,2,2) -- (5,2,3) -- (4,2,3) -- cycle;

	\draw[ fill={rgb,255:red,255; green,255; blue,255}, fill opacity=1] (2,2,3) -- (3,2,3) -- (3,2,4) -- (2,2,4) -- cycle;
	\draw[ fill={rgb,255:red,255; green,255; blue,255}, fill opacity=1] (2,1,4) -- (2,2,4) -- (3,2,4) -- (3,1,4) -- cycle;

	\draw[ fill={rgb,255:red,255; green,255; blue,255}, fill opacity=1] (3,2,3) -- (4,2,3) -- (4,2,4) -- (3,2,4) -- cycle;

    \draw[ fill={rgb,255:red,255; green,255; blue,255}, fill opacity=1] (5,1,3) -- (5,2,3) -- (5,2,4) -- (5,1,4) -- cycle;
	\draw[ fill={rgb,255:red,255; green,255; blue,255}, fill opacity=1] (4,2,3) -- (5,2,3) -- (5,2,4) -- (4,2,4) -- cycle;

	\draw[ fill={rgb,255:red,255; green,255; blue,255}, fill opacity=1] (3,2,4) -- (4,2,4) -- (4,2,5) -- (3,2,5) -- cycle;
	\draw[ fill={rgb,255:red,255; green,255; blue,255}, fill opacity=1] (3,1,5) -- (3,2,5) -- (4,2,5) -- (4,1,5) -- cycle;

    \draw[ fill={rgb,255:red,255; green,255; blue,255}, fill opacity=1] (5,1,4) -- (5,2,4) -- (5,2,5) -- (5,1,5) -- cycle;
	\draw[ fill={rgb,255:red,255; green,255; blue,255}, fill opacity=1] (4,2,4) -- (5,2,4) -- (5,2,5) -- (4,2,5) -- cycle;
	\draw[ fill={rgb,255:red,255; green,255; blue,255}, fill opacity=1] (4,1,5) -- (4,2,5) -- (5,2,5) -- (5,1,5) -- cycle;

	\draw[ fill={rgb,255:red,255; green,255; blue,255}, fill opacity=1](1,3,0) -- (2,3,0) -- (2,3,1) -- (1,3,1) -- cycle;
	\draw[ fill={rgb,255:red,255; green,255; blue,255}, fill opacity=1] (1,2,1) -- (1,3,1) -- (2,3,1) -- (2,2,1) -- cycle;

	\draw [ fill={rgb,255:red,255; green,255; blue,255}, fill opacity=1] (2,3,0) -- (3,3,0) -- (3,3,1) -- (2,3,1) -- cycle;

	\draw[ fill={rgb,255:red,255; green,255; blue,255}, fill opacity=1] (3,3,0) -- (4,3,0) -- (4,3,1) -- (3,3,1) -- cycle;

    \draw[ fill={rgb,255:red,255; green,255; blue,255}, fill opacity=1] (5,2,0) -- (5,3,0) -- (5,3,1) -- (5,2,1) -- cycle;
	\draw[ fill={rgb,255:red,255; green,255; blue,255}, fill opacity=1] (4,3,0) -- (5,3,0) -- (5,3,1) -- (4,3,1) -- cycle;

	\draw[ fill={rgb,255:red,255; green,255; blue,255}, fill opacity=1] (2,3,1) -- (3,3,1) -- (3,3,2) -- (2,3,2) -- cycle;
	\draw[ fill={rgb,255:red,255; green,255; blue,255}, fill opacity=1] (2,2,2) -- (2,3,2) -- (3,3,2) -- (3,2,2) -- cycle;

	\draw[ fill={rgb,255:red,255; green,255; blue,255}, fill opacity=1] (3,3,1) -- (4,3,1) -- (4,3,2) -- (3,3,2) -- cycle;

    \draw[ fill={rgb,255:red,255; green,255; blue,255}, fill opacity=1] (5,2,1) -- (5,3,1) -- (5,3,2) -- (5,2,2) -- cycle;
	\draw[ fill={rgb,255:red,255; green,255; blue,255}, fill opacity=1] (4,3,1) -- (5,3,1) -- (5,3,2) -- (4,3,2) -- cycle;

	\draw[ fill={rgb,255:red,255; green,255; blue,255}, fill opacity=1] (3,3,2) -- (4,3,2) -- (4,3,3) -- (3,3,3) -- cycle;
    \draw[ fill={rgb,255:red,255; green,255; blue,255}, fill opacity=1] (3,2,3) -- (3,3,3) -- (4,3,3) -- (4,2,3) -- cycle;
    
    \draw[ fill={rgb,255:red,255; green,255; blue,255}, fill opacity=1] (5,2,2) -- (5,3,2) -- (5,3,3) -- (5,2,3) -- cycle;
	\draw[ fill={rgb,255:red,255; green,255; blue,255}, fill opacity=1] (4,3,2) -- (5,3,2) -- (5,3,3) -- (4,3,3) -- cycle;
	\draw[ fill={rgb,255:red,255; green,255; blue,255}, fill opacity=1] (4,2,3) -- (4,3,3) -- (5,3,3) -- (5,2,3) -- cycle;

    \draw[ fill={rgb,255:red,255; green,255; blue,255}, fill opacity=1] (5,2,3) -- (5,3,3) -- (5,3,4) -- (5,2,4) -- cycle;
	\draw[ fill={rgb,255:red,255; green,255; blue,255}, fill opacity=1] (4,3,3) -- (5,3,3) -- (5,3,4) -- (4,3,4) -- cycle;
	\draw[ fill={rgb,255:red,255; green,255; blue,255}, fill opacity=1] (4,2,4) -- (4,3,4) -- (5,3,4) -- (5,2,4) -- cycle;

	\draw [ fill={rgb,255:red,255; green,255; blue,255}, fill opacity=1] (2,4,0) -- (3,4,0) -- (3,4,1) -- (2,4,1) -- cycle;
	\draw[ fill={rgb,255:red,255; green,255; blue,255}, fill opacity=1] (2,3,1) -- (2,4,1) -- (3,4,1) -- (3,3,1) -- cycle;

	\draw[ fill={rgb,255:red,255; green,255; blue,255}, fill opacity=1] (3,4,0) -- (4,4,0) -- (4,4,1) -- (3,4,1) -- cycle;
	\draw[ fill={rgb,255:red,255; green,255; blue,255}, fill opacity=1] (3,4,1) -- (3,4,1) -- (4,4,1) -- (4,3,1) -- cycle;

    \draw[ fill={rgb,255:red,255; green,255; blue,255}, fill opacity=1] (5,3,0) -- (5,4,0) -- (5,4,1) -- (5,3,1) -- cycle;
	\draw[ fill={rgb,255:red,255; green,255; blue,255}, fill opacity=1] (4,4,0) -- (5,4,0) -- (5,4,1) -- (4,4,1) -- cycle;
	\draw[ fill={rgb,255:red,255; green,255; blue,255}, fill opacity=1] (4,3,1) -- (4,4,1) -- (5,4,1) -- (5,3,1) -- cycle;

	\draw[ fill={rgb,255:red,255; green,255; blue,255}, fill opacity=1] (3,4,1) -- (4,4,1) -- (4,4,2) -- (3,4,2) -- cycle;
	\draw[ fill={rgb,255:red,255; green,255; blue,255}, fill opacity=1] (3,3,2) -- (3,4,2) -- (4,4,2) -- (4,3,2) -- cycle;

    \draw[ fill={rgb,255:red,255; green,255; blue,255}, fill opacity=1] (5,3,1) -- (5,4,1) -- (5,4,2) -- (5,3,2) -- cycle;
	\draw[ fill={rgb,255:red,255; green,255; blue,255}, fill opacity=1] (4,4,1) -- (5,4,1) -- (5,4,2) -- (4,4,2) -- cycle;
	\draw[ fill={rgb,255:red,255; green,255; blue,255}, fill opacity=1] (4,3,2) -- (4,4,2) -- (5,4,2) -- (5,3,2) -- cycle;

	\draw [ fill={rgb,255:red,255; green,255; blue,255}, fill opacity=1] (2,5,0) -- (3,5,0) -- (3,5,1) -- (2,5,1) -- cycle;
	\draw[ fill={rgb,255:red,255; green,255; blue,255}, fill opacity=1] (2,4,1) -- (2,5,1) -- (3,5,1) -- (3,4,1) -- cycle;

	\draw[ fill={rgb,255:red,255; green,255; blue,255}, fill opacity=1] (3,5,0) -- (4,5,0) -- (4,5,1) -- (3,5,1) -- cycle;
	\draw[ fill={rgb,255:red,255; green,255; blue,255}, fill opacity=1] (3,4,1) -- (3,5,1) -- (4,5,1) -- (4,4,1) -- cycle;

    \draw[ fill={rgb,255:red,255; green,255; blue,255}, fill opacity=1] (5,4,0) -- (5,5,0) -- (5,5,1) -- (5,4,1) -- cycle;
	\draw[ fill={rgb,255:red,255; green,255; blue,255}, fill opacity=1] (4,5,0) -- (5,5,0) -- (5,5,1) -- (4,5,1) -- cycle;
	\draw[ fill={rgb,255:red,255; green,255; blue,255}, fill opacity=1] (4,4,1) -- (4,5,1) -- (5,5,1) -- (5,4,1) -- cycle;

    \draw[ fill={rgb,255:red,255; green,255; blue,255}, fill opacity=1] (5,4,1) -- (5,5,1) -- (5,5,2) -- (5,4,2) -- cycle;
	\draw[ fill={rgb,255:red,255; green,255; blue,255}, fill opacity=1] (4,5,1) -- (5,5,1) -- (5,5,2) -- (4,5,2) -- cycle;
	\draw[ fill={rgb,255:red,255; green,255; blue,255}, fill opacity=1] (4,4,2) -- (4,5,2) -- (5,5,2) -- (5,4,2) -- cycle;

\end{tikzpicture}
}%
    \caption{Representation of $\CHN(f,4,4,4)$ with $f(x_1,x_2)=\left\lfloor \frac{2x_1}{\max(5-x_2,1)}\right\rfloor$}
    \label{fig:3d_CHN}
\end{figure}

The following game is a further generalization of Chocolate Game.

\begin{definition}[Multiple Changing Head Nim]
    Let $f_1,\ldots ,f_{\ell}$ 
    be functions from $\mathbb{Z}_{\ge 0}^k$ to $\mathbb{Z}_{\ge 0}$ for some nonnegative integer $k$. For $y_1,\ldots ,y_{\ell}$, $x_1,\ldots ,x_k$, the Multiple Changing Head Nim consists of $k$ piles $x_i$ for $1\le i\le k$ and ${\ell}$ piles whose sizes are $\min(f_j(x_1,\ldots  ,x_k),y_j)$ for $1\le j\le {\ell}$ which can be  changed  when we play on another pile. We denote this by $\MCHN(f_1,\ldots ,f_{\ell},x_1,\ldots  ,x_k,  y_1,\ldots ,y_{\ell})$.
\end{definition}

\begin{remark}
    Multiple Changing Head Nim is isomorphic to Changing Head Nim when $\ell=1$.
\end{remark}

\begin{definition}
\label{def:tastyhigher}
    Let $f$ be a function from $\mathbb{Z}_{\ge 0}$ to $\mathbb{Z}_{\ge 0}.$
    
    $\CHCG(f,x,y)$ is {\em tasty} if for every $0\le x' \le x$ and $0 \le y' \le y$, $$g(\CHCG(f,x',y'))  = x'\oplus \min(f(x'),y').$$
    
    Let $f_1,\ldots,f_\ell$ be $\ell$ functions from $\mathbb{Z}_{\ge 0}^k$ to $\mathbb{Z}_{\ge 0}.$
    
    Similary, $\CHN(f_1,x_1,\ldots,x_k,y_1)$ is {\em tasty} if for every $0\le x'_1 \le x_1, \ldots,0\le x'_k \le x_k,\ldots,0 \le y'_1 \le y_1$, $$g(\CHN(f_1,x'_1,\ldots,x'_k,y'_1))  = \min(f_1(x'_1,\ldots,x'_k),y'_1)\oplus x'_1\oplus\cdots\oplus x'_k.$$
    Similary, $\MCHN(f_1,\ldots,f_\ell,x_1,\ldots,x_k,y_1,\ldots,y_\ell)$ is {\em tasty} if for every $0\le x'_1 \le x_1, \ldots,0\le x'_k \le x_k,\ldots,0 \le y'_1 \le y_1,\ldots, 0 \le y'_\ell \le y_\ell$, 
    \begin{align*}
        &g(\MCHN(f_1,\ldots,f_\ell,x'_1,\ldots,x'_k,y'_1,\ldots,y'_\ell))  \\
        =& \min(f_1(x'_1,\ldots,x'_k),y'_1)\oplus\cdots\oplus \min(f_\ell(x'_1,\ldots,x'_k),y'_\ell)\oplus x'_1\oplus\cdots\oplus x'_k.
    \end{align*}
\end{definition}

\begin{definition}
    A function $f$ from $\mathbb{Z}_{\ge 0}^k$ to $\mathbb{Z}_{\ge 0}$ is monotonous according to all dimensions, if for all $i\le k$ and nonnegative integers $x_1,\ldots,x_k$ the function $F_{i,x_1,\ldots, x_k}(z)=f(x_1,\ldots,x_{i-1},zn,x_{i+1},\ldots,x_k)$ is monotonous.
\end{definition}





\section{Functions that satisfy the tasty condition}
\label{sec:tasty}

In this section, we show some results for the functions satisfying the tasty condition. First we show that some operators give us function satisfying the tasty condition if the entries satisfy the tasty condition. We also show some examples of functions satisfying the tasty condition and we show a form in which all functions satisfying the tasty condition can be written.

\begin{theorem}
\label{theorem:mintasty}
    For any functions $f_1,f_2$, if $f_1$ and $f_2$ satisfy the tasty condition then the function $F(z) = \min(f_1(z),f_2(z))$ satisfies the tasty condition.
\end{theorem}

\begin{proof}
    We take $z,z',i$, such that $\left\lfloor \frac{z}{2^i} \right\rfloor=\left\lfloor \frac{z'}{2^i} \right\rfloor$.

    Then, $\left\lfloor \frac{f_1(z)}{2^{i-1}} \right\rfloor=\left\lfloor \frac{f_1(z')}{2^{i-1}} \right\rfloor$ as $f_1$ satisfies the tasty condition and $\left\lfloor \frac{f_2(z)}{2^{i-1}} \right\rfloor=\left\lfloor \frac{f_2(z')}{2^{i-1}} \right\rfloor$ as $f$ satisfies the tasty condition.

    Thus, 
    \begin{align*}
        \left\lfloor \frac{\min(f_1(z),f_2(z))}{2^{i-1}} \right\rfloor=&\min\left(\left\lfloor \frac{f_1(z)}{2^{i-1}} \right\rfloor, \left\lfloor \frac{f_2(z)}{2^{i-1}} \right\rfloor\right) \\
        =&\min\left(\left\lfloor \frac{f_1(z')}{2^{i-1}} \right\rfloor, \left\lfloor \frac{f_2(z')}{2^{i-1}} \right\rfloor\right) \\
        =&\left\lfloor \frac{\min(f_1(z'),f_2(z'))}{2^{i-1}} \right\rfloor
    \end{align*}
     and  $F(z) = \min(f_1(z),f_2(z))$ satisfies the tasty condition.
\end{proof}

\begin{theorem}
\label{theorem:maxtasty}
    For any functions $f_1,f_2$, if $f_1$ and $f_2$ satisfy the tasty condition then the function $F(z) = \max(f_1(z),f_2(z))$ satisfies the tasty condition.
\end{theorem}

\begin{proof}
    We take $z,z',i$, such that $\left\lfloor \frac{z}{2^i} \right\rfloor=\left\lfloor \frac{z'}{2^i} \right\rfloor$.

    Then, $\left\lfloor \frac{f_1(z)}{2^{i-1}} \right\rfloor=\left\lfloor \frac{f_1(z')}{2^{i-1}} \right\rfloor$ as $f_1$ satisfies the tasty condition and $\left\lfloor \frac{f_2(z)}{2^{i-1}} \right\rfloor=\left\lfloor \frac{f_2(z')}{2^{i-1}} \right\rfloor$ as $f$ satisfies the tasty condition.

    Thus, 
    \begin{align*}
        \left\lfloor \frac{\max(f_1(z),f_2(z))}{2^{i-1}} \right\rfloor=&\max\left(\left\lfloor \frac{f_1(z)}{2^{i-1}} \right\rfloor, \left\lfloor \frac{f_2(z)}{2^{i-1}} \right\rfloor\right) \\
        =&\max\left(\left\lfloor \frac{f_1(z')}{2^{i-1}} \right\rfloor, \left\lfloor \frac{f_2(z')}{2^{i-1}} \right\rfloor\right) \\
        =&\left\lfloor \frac{\max(f_1(z'),f_2(z'))}{2^{i-1}} \right\rfloor
    \end{align*}
    
    and  $F(z) = \max(f_1(z),f_2(z))$ satisfies the tasty condition.
\end{proof}

\begin{theorem}
\label{theorem:tastyoperator}
    For any functions $f_1,f_2$ and for any operator $*$ applying an binary operation bit by bit and such as $0*0=0$, if $f_1$ and $f_2$ satisfy the tasty condition then the function $F(z) = f_1(z)* f_2(z)$ satisfies the tasty condition.
\end{theorem}

\begin{remark}
We can only consider operators such that $0*0=0$, because if not by applying bit by bit we obtain an infinite result as we can always extend the binary writing by adding zeros.    
\end{remark}

\begin{proof}
    We take $z,z',i$, such that $\left\lfloor \frac{z}{2^i} \right\rfloor=\left\lfloor \frac{z'}{2^i} \right\rfloor$.

    Then, $\left\lfloor \frac{f_1(z)}{2^{i-1}} \right\rfloor=\left\lfloor \frac{f_1(z')}{2^{i-1}} \right\rfloor$ as $f_1$ satisfies the tasty condition and $\left\lfloor \frac{f_2(z)}{2^{i-1}} \right\rfloor=\left\lfloor \frac{f_2(z')}{2^{i-1}} \right\rfloor$ as $f$ satisfies the tasty condition.

    Thus, as dividing by a power of $2$ is just moving the digits in the binary expression, when we divide a number calculated by our operator by a power of $2$, it is the same thing as dividing each of the component of the operator by the same value and thus, $$\left\lfloor \frac{f_1(z)* f_2(z)}{2^{i-1}} \right\rfloor=\left\lfloor \frac{f_1(z)}{2^{i-1}} \right\rfloor* \left\lfloor \frac{f_2(z)}{2^{i-1}} \right\rfloor=\left\lfloor \frac{f_1(z')}{2^{i-1}} \right\rfloor * \left\lfloor \frac{f_2(z')}{2^{i-1}} \right\rfloor=\left\lfloor \frac{f_1(z')* f_2(z')}{2^{i-1}} \right\rfloor$$ and  $F(z) = f_1(z)* f_2(z)$ satisfies the tasty condition.
\end{proof}

\begin{theorem}
\label{theorem:nimsumtasty}
    For any functions $f_1,f_2$, if $f_1$ and $f_2$ satisfy the tasty condition then the function $F(z) = f_1(z)\oplus f_2(z)$ satisfies the tasty condition.
\end{theorem}

\begin{proof}
As Nim--sum is an operator applying a binary operation bit by bit with $0\oplus0=0$, by Theorem~\ref{theorem:tastyoperator} $F(z) = f_1(z)\oplus f_2(z)$ satisfies the tasty condition.


\end{proof}

\begin{theorem}
\label{theorem:csttasty}
    For any nonnegative integer $m$,  the function $F(z) = m$ satisfies the tasty condition.
\end{theorem}

\begin{proof}
    We take $z,z',i$, such that $\left\lfloor \frac{z}{2^i} \right\rfloor=\left\lfloor \frac{z'}{2^i} \right\rfloor$.

    Then, $\left\lfloor \frac{m}{2^{i-1}} \right\rfloor=\left\lfloor \frac{m}{2^{i-1}} \right\rfloor$ and $F$ satisfies the tasty condition.
\end{proof}

\begin{theorem}
\label{theorem:floortasty}
    For any nonnegative integer $k$,  the function $F(z) = \left\lfloor \frac{z}{2k} \right\rfloor$ satisfies the tasty condition.
\end{theorem}

\begin{proof}
    We take $z,z',i$, such that $\left\lfloor \frac{z}{2^i} \right\rfloor=\left\lfloor \frac{z'}{2^i} \right\rfloor$.
    Then, $$\left\lfloor \frac{\frac{z}{2k}}{2^{i-1}} \right\rfloor=\left\lfloor \frac{z}{k2^{i}} \right\rfloor = \left\lfloor \left\lfloor\frac{1}{k}\left\lfloor \frac{z}{2^i} \right\rfloor \right\rfloor+ \frac{\frac{z}{2^i}-k\left\lfloor\frac{1}{k}\left\lfloor \frac{z}{2^i} \right\rfloor \right\rfloor}{k} \right\rfloor = \left\lfloor \frac{1}{k}\left\lfloor \frac{z}{2^i} \right\rfloor \right\rfloor \le \frac{1}{k}\left\lfloor \frac{z}{2^i} \right\rfloor=\frac{1}{k}\left\lfloor \frac{z'}{2^i} \right\rfloor$$ and $$\left\lfloor \frac{\frac{z}{2k}}{2^{i-1}} \right\rfloor=\left\lfloor \frac{z}{k2^{i}} \right\rfloor = \left\lfloor \left\lfloor\frac{1}{k}\left\lfloor \frac{z}{2^i} \right\rfloor \right\rfloor+ \frac{\frac{z}{2^i}-k\left\lfloor\frac{1}{k}\left\lfloor \frac{z}{2^i} \right\rfloor \right\rfloor}{k} \right\rfloor = \left\lfloor \frac{1}{k}\left\lfloor \frac{z}{2^i} \right\rfloor \right\rfloor > \frac{1}{k}\left\lfloor \frac{z}{2^i}\right\rfloor -1 =\frac{1}{k}\left\lfloor \frac{z'}{2^i}\right\rfloor -1.$$ Thus, $$\left\lfloor \frac{\frac{z}{2k}}{2^{i-1}} \right\rfloor=\left\lfloor\frac{1}{k}\left\lfloor \frac{z'}{2^i}\right\rfloor\right\rfloor$$ and similary  $$\left\lfloor \frac{\frac{z'}{2k}}{2^{i-1}} \right\rfloor=\left\lfloor\frac{1}{k}\left\lfloor \frac{z'}{2^i}\right\rfloor\right\rfloor.$$ Therefore,  $$\left\lfloor \frac{\frac{z}{2k}}{2^{i-1}} \right\rfloor=\left\lfloor \frac{\frac{z'}{2k}}{2^{i-1}} \right\rfloor$$ and $F$ satisfies the tasty condition.
\end{proof}

\begin{theorem}
\label{theorem:bounded_fitting}
    For any function $f$ from $\mathbb{Z}_{\ge 0}$ to $\mathbb{Z}_{\ge 0}$ bounded by $2^k-1$, let $F=2^{k}-f(z)-1$. If $f$ satisfies the tasty condition, then $F$ satisfies the tasty condition too.
\end{theorem}

\begin{proof}
By Theorem~\ref{theorem:csttasty}, $2^k-1$ satisfies the tasty condition.
By Theorem~\ref{theorem:nimsumtasty}, $F'(z)=(2^k-1)\oplus f(z)$ satisfies the tasty condition.
As $f$ is bounded by $2^k-1$, $(2^k-1)\oplus f(z)=2^{k}-f(z)-1$ and $F$ satisfies the tasty condition.
\end{proof}

\begin{lemma}
\label{lemma:nimsumconsecutive}
 For any positive integer $n$ and nonnegative integer $i$, $(n\oplus (n-1))+1$ is a power of two and $$\left\lfloor \frac{n\oplus(n-1)}{2^{i}} \right\rfloor=\max(2^{\log_2 ((n\oplus (n-1))+1)-i}-1,0).$$    
\end{lemma}

\begin{proof}
    For any nonnegative integer $k$, if the $k$-th bit of $n \oplus (n - 1)$ is $1$ in the binary expression, then the $k$-th bit of $n$ and $n-1$ in the binary expressions are different. In addition, the $k$-th bit of $n$ and $n-1$ in the binary expressions are  different only when it is either the least significant bit or in the previous bit of $n$ is $0$ and of $n-1$ is $1$. Therefore, $n \oplus (n - 1)$ is a chain of $1$ and $(n \oplus 
    (n - 1)) + 1$ is a power of two.
    

   We have $$\left\lfloor \frac{n\oplus(n-1)}{2^{i}} \right\rfloor\le \frac{n\oplus(n-1)}{2^{i}}=\frac{(n\oplus(n-1))+1}{2^{i}}-\frac{1}{2^{i}}=2^{\log ((n\oplus (n-1))+1)-i}-\frac{1}{2^{i}}$$ and $$\left\lfloor \frac{n\oplus(n-1)}{2^{i}} \right\rfloor> \frac{n\oplus(n-1)}{2^{i}}-1=\frac{(n\oplus(n-1))+1}{2^{i}}-\frac{1}{2^{i}}-1=2^{\log ((n\oplus (n-1))+1)-i}-\frac{1}{2^{i}}-1.$$ As $\max(2^{\log_2 ((n\oplus (n-1))+1)-i}-1,0)$ is an integer between the two, we have $$\left\lfloor \frac{n\oplus(n-1)}{2^{i}} \right\rfloor=\max(2^{\log_2 ((n\oplus (n-1))+1)-i}-1,0).$$
\end{proof}

\begin{lemma}
\label{lemma:xinoplus}
    For any $n$, for all $x\le \frac{n\oplus (n-1)}{2}$, there exist a positive integer $m$ and $\epsilon_{x,1},\ldots,\epsilon_{x,m} \in \{0,1\}$, which satisfy $$x=\bigoplus_{i=1}^m \epsilon_{x,i} \left\lfloor \frac{n\oplus(n-1)}{2^{i}} \right\rfloor.$$
\end{lemma}

\begin{proof}
    We prove this by induction on $x$.
    For $x=0$, we can take $m=1$ and $\epsilon_{x,1}=0$ giving us a Nim--sum which is $0$.
    
    We suppose that for all $x'<x$ the lemma holds.
    If $x\oplus (x-1)<x$ then we can take $$x\oplus (x-1)\oplus (x-1)=\bigoplus_{i=1}^m (\epsilon_{x\oplus (x-1),i}\oplus \epsilon_{(x-1),i}) \left\lfloor \frac{n\oplus(n-1)}{2^{i}} \right\rfloor.$$
    If $x\oplus (x-1)\ge x$ then by Lemma~\ref{lemma:nimsumconsecutive} $x\oplus (x-1)$ is a power of two and as $x\le \frac{n\oplus (n-1)}{2}$ then $x$ and $x-1$ have a shorter binary expression than $n\oplus (n-1)$ and as the Nim--sum is applied bit by bit we have $x\oplus (x-1)$ has a shorter binary expression than $n\oplus (n-1)$. Thus, we can find a $k\ge 1$ such as $$\left\lfloor \frac{n\oplus(n-1)}{2^{k}} \right\rfloor=2^{\log_2 ((n\oplus (n-1))+1)-k}-1=x\oplus (x-1)$$ then we just have to take $\epsilon_{x,k}=1-\epsilon_{x-1,k}$ and $\epsilon_{x,i}=\epsilon_{x-1,i}$ for $i\ne k$.

\end{proof}

\begin{lemma}
\label{lemma:inequalitytasty}
    If $f$ satisfies the tasty condition, then $$f(n)\oplus f(n-1)\le\frac{n\oplus (n-1)}{2}.$$
\end{lemma}

\begin{proof}
    If $n\oplus (n-1)<2^k$ for some $k$ then $\left\lfloor \frac{n}{2^{k}} \right\rfloor=\left\lfloor \frac{n-1}{2^{k}} \right\rfloor$ and then by tasty condition $\left\lfloor \frac{f(n)}{2^{k-1}} \right\rfloor=\left\lfloor \frac{f(n-1)}{2^{k-1}} \right\rfloor$ and $f(n)\oplus f(n-1)< 2^{k-1}$. 
    As by Lemma~\ref{lemma:nimsumconsecutive}, $n\oplus (n-1)+1$ is a power of $2$, there exists $k$ such that 
    $n \oplus (n-1) + 1= 2^k$ and $f(n) \oplus f(n-1) < 2^{k-1}.$
    Thus, $f(n)\oplus f(n-1)\le\frac{n\oplus (n-1)}{2}$.
\end{proof}

\begin{theorem}
    $f$ satisfies the tasty condition if and only if, for all nonnegative integer $n$, there exist nonnegative integers $m,k_1,\ldots,k_m,\ell_1,\ldots,\ell_m,c_1,\ldots,c_m$ such as $f(n')=F_n(n')$ for all $n'\le n$ and with $$F_n(z)=\bigoplus_{i=1}^m \max\left(\min\left(\left\lfloor \frac{z}{2^{k_i+1}} \right\rfloor,\ell_i+1\right),\ell_i\right)\oplus c_i.$$
\end{theorem}

\begin{proof}
    If for all $n$, $f(n')=F_n(n')$ for all $n'\le n$, by Theorem~\ref{theorem:mintasty}, Theorem~\ref{theorem:maxtasty}, Theorem~\ref{theorem:nimsumtasty}, Theorem~\ref{theorem:csttasty} and Theorem~\ref{theorem:floortasty}, we have that all $F_n$ satisfies the tasty condition then by taking $n$ large enough we have that $f$ satisfies the tasty condition.

    If $f$ satisfies the tasty condition, we prove by induction on $n$.
    For $n=0$, we can take $m=1$, $\ell_1=0$, $k_1=0$ and $c_1=f(0)$ giving us that $F_0(0)=f(0).$

    We suppose that we have $F_{n-1}$ such as $F_{n-1}(n')=f(n')$ for all $n'\le n-1$. 

    By Theorem~\ref{theorem:mintasty}, Theorem~\ref{theorem:maxtasty}, Theorem~\ref{theorem:nimsumtasty}, Theorem~\ref{theorem:csttasty} and Theorem~\ref{theorem:floortasty}, $F_{n-1}$ and $F_{n-1} \oplus f$ satisfy the tasty condition.
    We take $F_{n-1}(n)\oplus f(n)$, as $F_{n-1} \oplus f$ satisfy the tasty condition and $F_{n-1}(n-1)\oplus f(n-1)=0$, by Lemma~\ref{lemma:inequalitytasty} $$F_{n-1}(n)\oplus f(n)\le \frac{n\oplus (n-1)}{2}.$$
    By Lemma~\ref{lemma:xinoplus}, we can find a positive integer $m'$ and $\epsilon_{F_{n-1}(n)\oplus f(n),1},\ldots,\epsilon_{F_{n-1}(n)\oplus f(n),m'} \in \{0,1\}$, such that $$F_{n-1}(n)\oplus f(n)=\bigoplus_{i=1}^{m'} \epsilon_{F_{n-1}(n)\oplus f(n),i} \left\lfloor \frac{n\oplus(n-1)}{2^{i}} \right\rfloor.$$

    By adding to the Nim--sum of $F_{n-1}$ some terms for each $\epsilon_{F_{n-1}(n)\oplus f(n),i}=1$,  with $k_j=i-1$, $\ell_{j}=\left\lfloor \frac{n-1}{2^{i}} \right\rfloor$ and $c_j=\ell_j$, 
    we obtain $$F_n(z)=F_{n-1}(z)\bigoplus _{i=1}^{m'} \epsilon_{F_{n-1}(n)\oplus f(n),i} \left(\max\left(\min\left(\left\lfloor \frac{z}{2^{i}} \right\rfloor,\left\lfloor \frac{n-1}{2^{i}} \right\rfloor+1\right),\left\lfloor \frac{n-1}{2^{i}} \right\rfloor\right)\oplus \left\lfloor \frac{n-1}{2^{i}} \right\rfloor\right), $$
    since if $z<n$, $$F_n(z)=F_{n-1}(z)\bigoplus _{i=1}^{m'} \epsilon_{F_{n-1}(n)\oplus f(n),i} \left(\left\lfloor \frac{n-1}{2^{i}} \right\rfloor\oplus \left\lfloor \frac{n-1}{2^{i}} \right\rfloor\right)=F_{n-1}(z)=f(z)$$ and 
    \begin{eqnarray*}
    F_n(n)&=&F_{n-1}(n)\bigoplus _{i=1}^{m'} \epsilon_{F_{n-1}(n)\oplus f(n),i} \left( \left\lfloor \frac{n}{2^{i}} \right\rfloor\oplus \left\lfloor \frac{n-1}{2^{i}} \right\rfloor\right) \\&=&F_{n-1}(n)\bigoplus_{i=1}^{m'} \epsilon_{F_{n-1}(n)\oplus f(n),i} \left\lfloor \frac{n\oplus(n-1)}{2^{i}} \right\rfloor \\&=&f(n).
    \end{eqnarray*}
\end{proof}

\section{Main results}
\label{sec:main}


In this section, we prove that the tasty condition is a sufficient condition for having  Changing Height Chocolate Game is tasty, and by extension, tasty condition is also a sufficient condition for having Multiple Changing Head Nim is tasty.

We also prove that the tasty condition is a necessary condition if the function are monotonous according to each dimension for having all Multiple Changing Head Nim using these functions tasty.

We first prove the easier case,  Changing Height Chocolate Game, as the proof is almost the same one as for Multiple Changing Head Nim but making the notation easier to understand and all equations shorter. If the reader understand the first two subsection, the other subsections are just the same proof but for Multiple Changing Head Nim.

\subsection{Sufficient Condition}

In this subsection, we prove that the tasty condition is a sufficient condition to have $g(\CHCG(f,n,m))=\min(m,f(n))\oplus n$.
Let $f$ be a function that satisfies the tasty condition.

\begin{lemma}
\label{lemma:yf(x)f(x-1)}
    For all $x,y \in \mathbb{Z}_{\ge 0}$, 
    $$x<y \implies x\oplus f(y) \oplus f(y-1) <y.$$
\end{lemma}

\begin{proof}
     Let $x,y \in \mathbb{Z}_{\ge 0}$ and $y<x$.
    We take the minimal $i$ such that $$ \left\lfloor \dfrac{f(y-1)}{2^{i}} \right\rfloor = \left\lfloor \dfrac{f(y)}{2^{i}} \right\rfloor,$$ then $$f(y)\oplus f(y-1)<2^{i}$$ and $$\left\lfloor \dfrac{x}{2^{i}} \right\rfloor = \left\lfloor \dfrac{x\oplus f(y) \oplus f(y-1)}{2^{i}} \right\rfloor.$$
    If $x\oplus f(y) \oplus f(y-1)\ge y$, then we have $$x\oplus f(y) \oplus f(y-1)\ge y>y-1\ge x,$$ which implies $$\left\lfloor \dfrac{y}{2^{i}} \right\rfloor = \left\lfloor \dfrac{y-1}{2^{i}} \right\rfloor.$$ However, by tasty condition, we have $$\left\lfloor \dfrac{f(y-1)}{2^{i-1}} \right\rfloor = \left\lfloor \dfrac{f(y)}{2^{i-1}} \right\rfloor,$$ contradicting the minimality of $i$.
    Thus, $x\oplus f(y) \oplus f(y-1) < y$.
\end{proof}

\begin{lemma}
\label{lemma:setequality}
    For any $y \in \mathbb{Z}_{\ge 0}$, we have $$\{f(y-1)\oplus x \mid 0\le x < y\}=\{f(y)\oplus x\mid0\le x < y\}.$$
\end{lemma}

\begin{proof}
    For every $z \in \{f(y-1)\oplus x\mid0\le x < y\},$ there exists $z'$ such that $0\le z'<y$ and $z=z'\oplus f(y-1)$. Then, by Lemma~\ref{lemma:yf(x)f(x-1)}, $0\le z'\oplus f(y-1)\oplus f(y)<y$  and $z = f(y) \oplus (z'\oplus f(y-1)\oplus f(y))\in \{f(y)\oplus x\mid0\le x < y\}$.

    For every $z \in \{f(y)\oplus x\mid0\le x < y\},$ there exists $z'$ such that $0\le z'<y$ and $z=z'\oplus f(y-1)$. Then, by Lemma~\ref{lemma:yf(x)f(x-1)}, $0\le z'\oplus f(y-1)\oplus f(y)<y$  and $z = f(y-1) \oplus (z'\oplus f(y-1)\oplus f(y)) \in \{f(y-1)\oplus x\mid0\le x < y\}$. 
    
    Therefore, $$\{f(y-1)\oplus x \mid 0\le x < y\}=\{f(y)\oplus x\mid0\le x < y\}$$
    holds.
\end{proof}

\begin{lemma}
\label{lemma:equality_f(x)f(y)}
    For any $y \in \mathbb{Z}_{\ge 0}$, we have
    $$\{f(x)\oplus x\mid0\le x < y\}=\{f(y)\oplus x\mid0\le x < y\}.$$
\end{lemma}

\begin{proof}
    We prove this by induction on $y$.
    It is trivial that when $y=0$, the both sets are empty sets and the equation holds. From the induction hypothesis, we have 
    \begin{align*}
\{f(x)\oplus x\mid0\le x < y\}&=\{f(y-1)\oplus x\mid0\le x < y-1\} \cup \{f(y-1)\oplus(y-1)\}  \\
    &=\{f(y-1)\oplus x\mid0\le x < y\}.
    \end{align*}
    By Lemma~\ref{lemma:setequality}, $$\{f(y-1)\oplus x\mid0\le x < y\}=\{f(y)\oplus x\mid0\le x < y\}.$$
    Therefore, for any $y \in \mathbb{Z}_{\ge 0}$, $$\{f(x)\oplus x\mid0\le x < y\}=\{f(y)\oplus x\mid0\le x < y\}$$ holds. 
    
\end{proof}

\begin{theorem}
\label{th:sufficient}
   For any $n,m \in \mathbb{Z}_{\ge 0}$, we have $$g(\CHCG(f,n,m))=\min(m,f(n))\oplus n.$$
\end{theorem}

\begin{proof}
    We prove this by induction on $n$.

    If $m\le f(n)$, then $\CHCG(f,n,m)$ is isomorphic to $ \SCG(n,m)$ and $$g(\CHCG(f,n,m))=g(\SCG(n,m))=m\oplus n= \min(m,f(n))\oplus n. $$ 

    If $m > f(n)$, 
    \begin{align*}
        g(\CHCG(f,n,m))&= {\rm mex} (\{g(\CHCG(f,n',m)), g(\CHCG(f,n,m'))\mid 0\le n'< n, 0 \le m'<f(n) \}) \\
        &={\rm mex} (\{\min(f(n'),m)\oplus n', m' \oplus n\mid 0\le n'< n, 0 \le m'<f(n) \}).
    \end{align*}
    However, by Lemmas~\ref{theorem:mintasty} and ~\ref{lemma:equality_f(x)f(y)}, 
    \begin{align*}
        &\{\min(f(n'),m)\oplus n', m' \oplus n\mid 0\le n'< n, 0 \le m'<f(n) \} \\
        =&\{\min(f(n),m)\oplus n', m' \oplus n\mid 0\le n'< n, 0 \le m'<f(n) \}
    \end{align*} and then 
    \begin{align*}
        g(\CHCG(f,n,m))&={\rm mex} (\{f(n)\oplus n', m' \oplus n\mid 0\le n'< n, 0 \le m'<f(n) \}) \\
        &=g(\SCG(n,f(n))=f(n)\oplus n.
    \end{align*}
\end{proof}

\subsection{Necessary Condition}
Let $f$ be an increasing chocolate function or a decreasing chocolate function. 
In this subsection, we prove that the tasty condition is a necessary condition to have $g(\CHCG(f,n,m))=\min(m,f(n))\oplus n$.

Note that since $f$ is an increasing chocolate function or a decreasing chocolate function, $\CHCG(f,n,m)$ is $\CB(f, n, m)$ or $\DCG(f, n, m).$





\begin{lemma}
\label{lemma:inneguality}
    For any $z,z' \in \mathbb{Z}_{\ge 0}$, if $z\ne z'$ then $z\oplus f(z) \ne z'\oplus f(z')$.
\end{lemma}

\begin{proof}
    Let $z,z'  \in \mathbb{Z}_{\ge 0}$ with $z\ne z'$. Suppose without loss of generality that $z>z'$.
    As for all $n,m \in \mathbb{Z}_{\ge 0}$,  $g(\CHCG(f,n,m))=\min(m,f(n))\oplus n$, if we take $m>\max(f(z),f(z'))$ then we have $g(\CHCG(f,z,m))=f(z)\oplus z$ and $g(\CHCG(f,z',m))=f(z')\oplus z'$. However, there is a move from $\CHCG(f,z,m)$ to $\CHCG(f,z',m).$ Then the Sprague--Grundy values must be different and $f(z)\oplus z\ne f(z')\oplus z'$.
\end{proof}

\begin{lemma}
 \label{lemma:divide_by_two}
    If for $x,k,m \in \mathbb{Z}_{\ge 0}$, there exists $y$ such that for all $x2^k\le z<(x+1)2^k$ we have $y2^k\le \min(m,f(z))<(y+1)2^k$, then 
    \begin{itemize}
        \item either for all $x2^k\le z'<(x+1)2^k$, $(2y)2^{k-1}\le \min(m,f(z'))<(2y+1)2^{k-1}$,
        \item or for all $x2^k\le z'<(x+1)2^k$, $(2y+1)2^{k-1}\le \min(m,f(z'))<(2y+2)2^{k-1}$.
    \end{itemize}
\end{lemma}

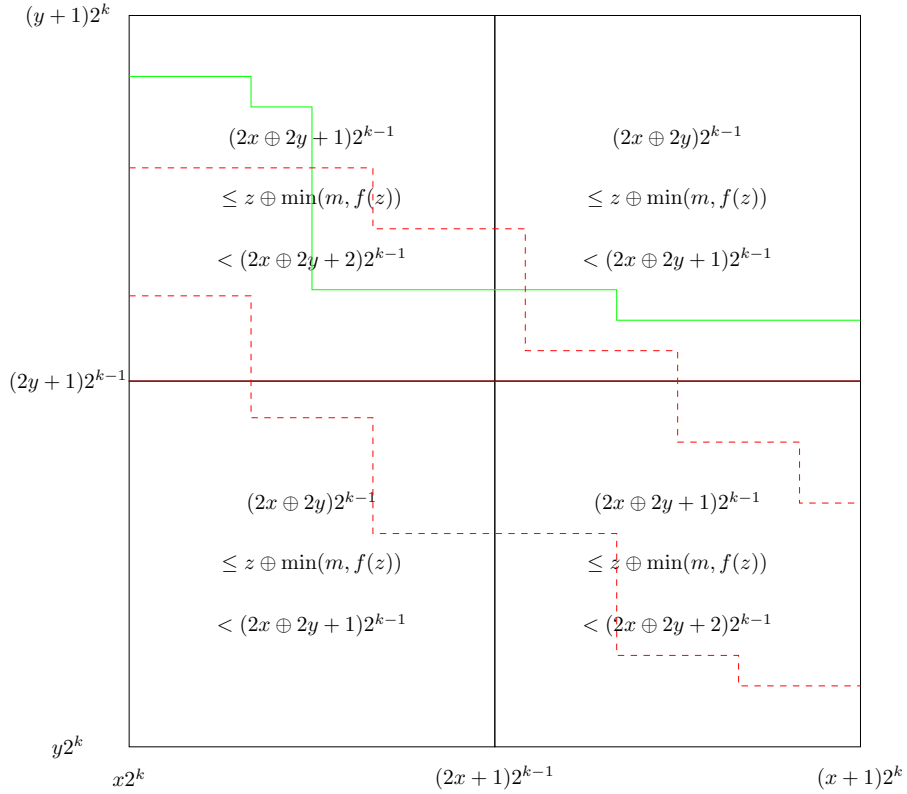
\begin{figure}[!ht]
\centering
\resizebox{\textwidth}{!}{%
\begin{tikzpicture}
  \node at (0,-0.5) {$x2^k$};

  \node at (6,-0.5) {$(2x+1)2^{k-1}$};

  \node at (12,-0.5) {$(x+1)2^k$};

  \node at (-1,0) {$y2^k$};
  \node (node1) at (-1,6) {$(2y+1)2^{k-1}$};
  \node at (-1,12) {$(y+1)2^k$};

\draw  (0,0) rectangle (6,6);
\draw  (0,6) rectangle (6,12);
\draw  (6,6) rectangle (12,12);
\draw  (6,0) rectangle (12,6);
\node at (3,4) {$
    (2x\oplus 2y)2^{k-1}$};
\node at (3,3) {
    $\le z\oplus \min(m,f(z))$  
   };
\node at (3,2) {  
    $<(2x\oplus 2y+1)2^{k-1}$};
\node at (9,4) {
    $(2x\oplus 2y+1)2^{k-1} $};
\node at (9,3) {
    $\le z\oplus \min(m,f(z))$ };
\node at (9,2) {
    $<(2x\oplus 2y+2)2^{k-1}$};
\node at (3,10) {
   $ (2x\oplus 2y+1)2^{k-1} $
};
\node at (3,9) {
    $\le z\oplus \min(m,f(z))$  
};
\node at (3,8) {
    $<(2x\oplus 2y+2)2^{k-1}$
};
\node at (9,10) {
    $(2x\oplus 2y)2^{k-1}$ 
};
\node at (9,9) { 
    $\le z\oplus \min(m,f(z))$ 
};
\node at (9,8) {
    $<(2x\oplus 2y+1)2^{k-1}$
};

  \draw[red] (12,6) -- (0,6);

  \draw[green] (0,11) -- (2,11) -- (2,10.5) -- (3,10.5) -- (3,7.5) -- (8,7.5) --(8,7) -- (12,7) ;

  \draw[red,dashed] (0,9.5) -- (4,9.5) -- (4,8.5) -- (6.5,8.5) -- (6.5,6.5) -- (9,6.5) -- (9,5) -- (11,5) -- (11,4) -- (12,4);

  \draw[red,dashed] (0,7.4) -- (2,7.4) --(2,5.4) -- (4,5.4) -- (4,3.5) -- (8,3.5) -- (8,1.5) -- (10,1.5) -- (10,1) --(12,1);
\end{tikzpicture}
}%
\caption{Graphic representation of the proof of Lemma~\ref{lemma:divide_by_two}}
\label{fig:chocodiagonal}
\end{figure}

Lemma~\ref{lemma:divide_by_two} can be understood as if $\min(m,f(z))$ stays in a box of born $y2^k$ and $(y+1)2^k$ between $x2^k$ and $(x+1)2^k$ then $\min(m,f(z))$ never cross the line separating the upper half and the lower half.
Figure \ref{fig:chocodiagonal} shows this intuition. The red dotted lines represent shapes of chocolates which cannot be tasty, and the green line represents a shape of chocolate which can be tasty.

We can understand the proof as the upper left quarter and the lower right quarter take values in the same set which only contain $2^{k-1}$ values and as $f$ is monotonous we have to pass entirely by one of these two quarters. Thus, as all values of $\min(m,f(z))$ have to be different, when we have pass entirely by one quarter by pigeon hole principle, we cannot enter in the second one.

The rigorous proof is as follows:
\begin{proof}
    We take $x,k,m,y \in \mathbb{Z}_{\ge 0}$ as for all $x2^k\le z<(x+1)2^k,$ we have $y2^k\le \min(m,f(z))<(y+1)2^k$.

    We have two possibilities for $z$ such that $$x2^k\le z < (2x+1)2^{k-1} \text{ and } (2x+1)2^{k-1}\le z < (x+1)2^{k}.$$ We also have two possibilities for $\min(m,f(z))$ such that $$y2^{k} \le \min(m,f(z)) < (2y+1)2^{k-1} \text{ and } (2y+1)2^{k-1} \le \min(m,f(z)) < (y+1)2^{k}.$$ Among the possibilities for $z$ and $\min(m,f(z))$, we note that  for two cases we have the same condition on $z\oplus \min(m,f(z))$.


     If $x2^k\le z < (2x+1)2^{k-1}$ and $(2y+1)2^{k-1} \le \min(m,f(z)) < (y+1)2^{k}$, then as all $z$ have the same binary writing if we exclude the $k-1$ least significant bits and all $\min(m,f(z))$ have the same binary writing if we exclude the $k-1$ least significant bits,
     $$((2x\oplus 2y)+1)2^{k-1}\le z\oplus \min(m,f(z))<((2x\oplus 2y)+2)2^{k-1}.$$
Note that we used $x2^k = (2x)2^{k-1}$ and $(2x) \oplus (2y + 1) = (2x \oplus 2y)  + 1.$
    
    Also, if $(2x+1)2^{k-1}\le z < (x+1)2^{k}$ and $y2^{k} \le \min(m,f(z)) < (2y+1)2^{k-1}$, then as all $z$ have the same binary writing if we exclude the $k-1$ least significant bits and all $\min(m,f(z))$ have the same binary writing if we exclude the $k-1$ least significant bits,
    $$((2x\oplus 2y)+1)2^{k-1}\le z\oplus \min(m,f(z))<((2x\oplus 2y)+2)2^{k-1}.$$
    
    Without loss of generality, we assume that $f$ is decreasing. For the case $f$ is increasing, which is already shown in \cite{MN21}, we can prove this by using a similar argument.\\
    

     Suppose that $\min(m,f((2x+1)2^{k-1}))<(2y+1)2^{k-1}$, then as $f$ is decreasing, for all $(2x+1)2^{k-1}\le z<(x+1)2^{k}$, $y2^{k} \le \min(m,f(z)) < (2y+1)2^{k-1}$ and as mentioned above, $$((2x\oplus 2y)+1)2^{k-1}\le z\oplus \min(m,f(z))<((2x\oplus 2y)+2)2^{k-1}.$$ 

However, there are $2^{k-1}$ different values between $((2x\oplus 2y)+1)2^{k-1}$ and $((2x\oplus 2y)+2)2^{k-1}-1$ and there are also exactly $2^{k-1}$ different possible values for $z$. 

For all $z,z'$, $z\oplus \min(m,f(z)) \ne z'\oplus \min(m,f(z'))$ by Lemma~\ref{lemma:inneguality}. Then as we already have all these values between $(2y)2^{k-1}$ and $(2y+1)2^{k-1}-1$, there is no other $z$ such as $((2x\oplus 2y)+1)2^{k-1}\le z\oplus \min(m,f(z))<((2x\oplus 2y)+2)2^{k-1}$. 

Then there is no $z$ such as $x2^k\le z < (2x+1)2^{k-1}$ and $(2y+1)2^{k-1} \le \min(m,f(z)) < (y+1)2^{k}$ which means that $\min(m,f(x2^{k}))< (2y+1)2^{k-1}$. Then as $f$ is decreasing, for all $x2^k\le z'<(x+1)2^k$, $(2y)2^{k-1}\le \min(m,f(z'))<(2y+1)2^{k-1}.$ \\

    Suppose that $\min(m,f((2x+1)2^{k-1}))\ge(2y+1)2^{k-1}$, then as $f$ is decreasing, for all $x2^k\le z<(2x+1)2^{k-1}$, $(2y+1)2^{k-1} \le \min(m,f(z)) < (y+1)2^{k}$ and as mentioned above,  $$((2x\oplus 2y)+1)2^{k-1}\le z\oplus \min(m,f(z))<((2x\oplus 2y)+2)2^{k-1}.$$ 
    
    However, there are $2^{k-1}$ different values between $((2x\oplus 2y)+1)2^{k-1}$ and $((2x\oplus 2y)+2)2^{k-1}-1$ and there are also exactly $2^{k-1}$ different possible values for $z$. 
    
    As for all $z,z'$, $z\oplus \min(m,f(z)) \ne z'\oplus \min(m,f(z'))$ by Lemma~\ref{lemma:inneguality}. Then as we already have all these value between $(2y+1)2^{k-1}$ and $(2y+2)2^{k-1}-1$ there is no other $z$ such as $$((2x\oplus 2y)+1)2^{k-1}\le z\oplus \min(m,f(z))<((2x\oplus 2y)+2)2^{k-1}.$$
    
    Then there is no $z$ such as $(2x+1)2^{k-1}\le z < (x+1)2^{k}$ and $y2^{k} \le \min(m,f(z)) < (2y+1)2^{k-1}$ which means that $\min(m,f((x+1)2^{k}-1))\ge (2y+1)2^{k-1}$. Then as $f$ is decreasing, $x2^k\le z'<(x+1)2^k$, $(2y+1)2^{k-1}\le \min(m,f(z'))<(2y+2)2^{k-1}$.

\end{proof}


\begin{lemma}
\label{lemma:bornonf(z)}
    For all $x,k$ in $\mathbb{Z}_{\ge 0},$ there exists $y \in \mathbb{Z}_{\ge 0}$ such that for all $x2^k\le z<(x+1)2^k$, we have $y2^{k-1}\le \min(m,f(z))<(y+1)2^{k-1}$.
\end{lemma}

\begin{proof}
    
We take $x,k \in \mathbb{Z}_{\ge 0}$.
We use decreasing induction on $k$.

If $k>1+\log_2(m)$, then for all $z$, $\min(m,f(z))\le2^{\log_2(m)}\le2^{k-1}$. Then if we take $y=0$, we have for all $z$, $y2^{k-1}=0\le \min(m,f(z))<1\times2^{k-1}$.

If $k\le 1+\log_2(m)$ and the statement holds for $k+1$, then there exists $y'$ in $\mathbb{Z}_{\ge 0}$ such that for all $\left\lfloor \dfrac{x}{2} \right\rfloor 2^{k+1}\le z<\left(\left\lfloor \dfrac{x}{2}\right \rfloor+1\right)2^{k+1}$, we have $y'2^{k}\le \min(m,f(z))<(y'+1)2^{k}$. If $x2^{k}\le z<(x+1)2^{k}$ then $\left \lfloor \dfrac{x}{2} \right\rfloor 2^{k+1}\le z<\left(\left \lfloor \dfrac{x}{2} \right\rfloor+1\right)2^{k+1}$ and by Lemma~\ref{lemma:divide_by_two}, either for all $x2^{k}\le z<(x+1)2^{k}$,  $(2y')2^{k-1}\le \min(m,f(z))<(2y'+1)2^{k-1}$ or for all $x2^{k}\le z<(x+1)2^{k}$, $(2y'+1)2^{k-1}\le \min(m,f(z))<(2y'+2)2^{k-1}$ giving in both case what we want.

By decreasing induction, the statement holds. 
\end{proof}

\begin{theorem}
\label{th:necessary}
    $f$ satisfies the tasty condition.
\end{theorem}

\begin{proof}


     We take $z,z' \in \mathbb{Z}_{\ge 0} $ and an integer $i$ such that $\left\lfloor \dfrac{z}{2^i} \right\rfloor=\left\lfloor \dfrac{z'}{2^i} \right\rfloor$.

    By Lemma~\ref{lemma:bornonf(z)} and by taking $m>\max(\{f(x)\mid0\le x \le \max(z,z')\})$, there exists $y$ such that for all $\left\lfloor \dfrac{z}{2^i} \right\rfloor2^i\le z''<\left(\left\lfloor \dfrac{z}{2^i} \right\rfloor+1\right)2^i$, $$y2^{i-1}\le f(z'')<(y+1)2^{i-1}.$$ We have $\left\lfloor \dfrac{z}{2^i} \right\rfloor2^i\le z,z'<\left(\left\lfloor \dfrac{z}{2^i} \right\rfloor+1\right)2^i$ as $\left\lfloor \dfrac{z}{2^i} \right\rfloor=\left\lfloor \dfrac{z'}{2^i} \right\rfloor$. Then $y2^{i-1}\le f(z),f(z')<(y+1)2^{i-1}$ and $ \left\lfloor \dfrac{h(z)}{2^{i-1}} \right\rfloor=\left\lfloor \dfrac{h(z')}{2^{i-1}} \right\rfloor=y.$ 
    
\end{proof}

\begin{theorem}
\label{th:necessary_and_sufficientCB}
For all $x,y$ in $\mathbb{Z}_{\ge 0}$, $g(\CHCG(f, x, y))=g(\SCG(x,y))$ if $f$ satisfies the tasty condition.

    For all monotonous function $f$, $g(\CHCG(f, x, y))=g(\SCG(x,y))$ for all nonnegative integers $x,y$ if and only if $f$ satisfies the tasty condition.
\end{theorem}

\begin{proof}
    By Theorem~\ref{th:necessary}, if $f$ is monotonous, the tasty condition is a necessary condition to have for all $x,y$ in $\mathbb{Z}_{\ge 0}$, $g(\CHCG(f, x, y))=g(\SCG(x,y))$.

    By Theorem~\ref{th:sufficient}, the tasty condition is a sufficient condition to have for all $x,y$ in $\mathbb{Z}_{\ge 0}$, $g(\CHCG(f, x, y))=g(\SCG(x,y))$. 
\end{proof}

Next, we show that, if an increasing chocolate bar is tasty, there exists a corresponding decreasing chocolate bar, and the cross-sections of the two pieces of chocolate fit together perfectly to form a single rectangle.

\begin{theorem}
\label{thm:fitting}
    For any increasing chocolate function $f$ and nonnegative integer $x$, let     $F_x(z) = 2^{\lceil \log_2 f(x) \rceil}-f(z)-1$. If $\CB(f,x,y)$ is tasty for all nonnegative integers $x,y$, then $\DCG(F_x,x,y)$ is tasty for all nonnegative integers $x,y$.
\end{theorem}

\begin{proof}
    If for all $x,y$ in $\mathbb{Z}_{\ge 0}$, $\CB(f,x,y)$ is tasty then by Theorem~\ref{th:necessary_and_sufficientCB}, $f$ satisfies the tasty condition.

    Let $f'(z)=\min(f(z),f(x))$ and $F'_x(z)=2^{\lceil \log_2 \max(f') \rceil}-f'(z)-1$ then $f'$ is bounded and by Theorem~\ref{theorem:mintasty} satisfies the tasty condition and by Theorem~\ref{theorem:bounded_fitting}, $F'_x$ satisfies the tasty condition too.
    
    By Theorem~\ref{th:necessary_and_sufficientCB}, $\DCG(F'_x,x,y)$ is tasty but for $z\le x$, $\max(2^{\lceil \log_2 f(x) \rceil}-f(z)-1, 2^{\lceil \log_2 f(x) \rceil}-f(x)-1)=2^{\lceil \log_2 f(x) \rceil}-f(z)-1$ and then $\DCG(F_x,x,y)$ is tasty.
\end{proof}

\subsection{Sufficient Condition with many functionnal dimension}

In this subsection, we prove that the tasty condition is a sufficient condition to have \begin{align*}
    &g(\MCHN(f_1,\ldots  ,f_{\ell},n_1,\ldots  ,n_k,m_1,\ldots  ,m_{\ell})) \\
    &=\min(m_1,f_1(n_1,\ldots  ,n_k))\oplus\cdots  \oplus \min(m_{\ell},f_{\ell}(n_1,\ldots  ,n_k))\oplus n_1\oplus\cdots  \oplus n_k.
\end{align*}
In this subsection, we take functions $f_1,\ldots  ,f_{\ell}$ that satisfy the tasty condition.

\begin{lemma}
\label{lemma:equality_f(x)f(y)multif}
    For any $y,i \in \mathbb{Z}_{\ge 0}$ and $m_1,...,m_{\ell},x_1,\ldots  ,x_{i-1},x_{i+1},\ldots  ,x_k$, 
    \begin{align*}
        &\{\min(m_1,f_1(x_1,\ldots  ,x_i,\ldots  ,x_k))\oplus \cdots   \oplus \min(m_{\ell},f_{\ell}(x_1,\ldots  ,x_i,\ldots  ,x_k)) \\
        &\oplus x_1 \oplus \cdots   \oplus x_i \oplus \cdots   \oplus x_k\mid0\le x_i < y\} \\
        =&\{\min(m_1,f_1(x_1,\ldots  ,y,\ldots  ,x_k))\oplus \cdots   \oplus \min(m_{\ell},f_{\ell}(x_1,\ldots  ,y,\ldots  ,x_k)) \\
        &\oplus x_1 \oplus \cdots   \oplus x_i \oplus \cdots   \oplus x_k\mid0\le x_i < y\}
    \end{align*}
\end{lemma}

\begin{proof}

    We take $y,i \in \mathbb{Z}_{\ge 0}$ and $m_1,...,m_{\ell},x_1,\ldots  ,x_{i-1},x_{i+1},\ldots  ,x_k.$

    By Lemma~\ref{lemma:equality_f(x)f(y)},  Theorem~\ref{theorem:mintasty} and Theorem~\ref{theorem:nimsumtasty}, 
    \begin{align*}
        &\{\min(m_1,f_1(x_1,\ldots  ,x_i,\ldots  ,x_k))\oplus \cdots   \oplus \min(m_{\ell},f_{\ell}(x_1,\ldots  ,x_i,\ldots  ,x_k))  \oplus x_i \mid0\le x_i < y\} \\
        =&\{\min(m_1,f_1(x_1,\ldots  ,y,\ldots  ,x_k)\oplus \cdots   \oplus \min(m_{\ell},f_{\ell}(x_1,\ldots  ,y,\ldots  ,x_k)) \oplus x_i \mid0\le x_i < y\}
    \end{align*}

    Then, for any $0\le x_i < y$, there exists $0 \le x'_i < y$ such that

    $\min(m_1,f_1(x_1,\ldots  ,x_i,\ldots  ,x_k))\oplus \cdots   \oplus \min(m_{\ell},f_{\ell}(x_1,\ldots  ,x_i,\ldots  ,x_k))  \oplus x_i=\min(m_1,f_1(x_1,\ldots  ,y,\ldots  ,x_k))\oplus \cdots   \oplus \min(m_{\ell},f_{\ell}(x_1,\ldots  ,y,\ldots  ,x_k)) \oplus x'_i$. Thus, we have 
    \begin{align*}
        &\min(m_1,f_1(x_1,\ldots  ,x_i,\ldots  ,x_k))\oplus \cdots   \oplus \min(m_{\ell},f_{\ell}(x_1,\ldots  ,x_i,\ldots  ,x_k)) \oplus x_1 \oplus \cdots   \oplus x_i \oplus \cdots   \oplus x_k \\
        =&\min(m_1,f_1(x_1,\ldots  ,y,\ldots  ,x_k))\oplus \cdots   \oplus \min(m_{\ell},f_{\ell}(x_1,\ldots  ,y,\ldots  ,x_k)) \oplus x_1 \oplus \cdots   \oplus x'_i \oplus \cdots   \oplus x_k
    \end{align*}
    giving us that 
    \begin{align*}
        &\{\min(m_1,f_1(x_1,\ldots  ,x_i,\ldots  ,x_k))\oplus \cdots   \oplus \min(m_{\ell},f_{\ell}(x_1,\ldots  ,x_i,\ldots  ,x_k)) \\
      &\oplus x_1 \oplus \cdots   \oplus x_i \oplus \cdots   \oplus x_k\mid0\le x_i < y\} \\
        \subseteq &\{\min(m_1,f_1(x_1,\ldots  ,y,\ldots  ,x_k))\oplus \cdots   \oplus \min(m_{\ell},f_{\ell}(x_1,\ldots  ,y,\ldots  ,x_k)) \\
        &\oplus x_1 \oplus \cdots   \oplus x_i \oplus \cdots   \oplus x_k\mid0\le x_i < y\}.
    \end{align*}
    By a similar way,
    \begin{align*}
        &\{\min(m_1,f_1(x_1,\ldots  ,x_i,\ldots  ,x_k))\oplus \cdots   \oplus \min(m_{\ell},f_{\ell}(x_1,\ldots  ,x_i,\ldots  ,x_k)) \\
        &\oplus x_1 \oplus \cdots   \oplus x_i \oplus \cdots   \oplus x_k\mid0\le x_i < y\} \\
        \supseteq &\{\min(m_1,f_1(x_1,\ldots  ,y,\ldots  ,x_k))\oplus \cdots   \oplus \min(m_{\ell},f_{\ell}(x_1,\ldots  ,y,\ldots  ,x_k)) \\
        &\oplus x_1 \oplus \cdots   \oplus x_i \oplus \cdots   \oplus x_k\mid0\le x_i < y\}.
    \end{align*}
    Thus, we have
    \begin{align*}
        &\{\min(m_1,f_1(x_1,\ldots  ,x_i,\ldots  ,x_k))\oplus \cdots   \oplus \min(m_{\ell},f_{\ell}(x_1,\ldots  ,x_i,\ldots  ,x_k)) \\
        &\oplus x_1 \oplus \cdots   \oplus x_i \oplus \cdots   \oplus x_k\mid0\le x_i < y\} \\
        = &\{\min(m_1,f_1(x_1,\ldots  ,y,\ldots  ,x_k))\oplus \cdots   \oplus \min(m_{\ell},f_{\ell}(x_1,\ldots  ,y,\ldots  ,x_k)) \\
        &\oplus x_1 \oplus \cdots   \oplus x_i \oplus \cdots   \oplus x_k\mid0\le x_i < y\}.
    \end{align*}
\end{proof}

\begin{theorem}
\label{th:sufficientmultif}
   For any $n_1, \ldots  , n_k$, $m_1, \ldots  , m_{\ell} \in \mathbb{Z}_{\ge 0}$,  
   \begin{align*}
      &g(\MCHN(f_1,\ldots  ,f_{\ell},n_1,\ldots  ,n_k,m_1,\ldots  ,m_{\ell})) \\
       =&\min(m_1,f_1(n_1,\ldots  ,n_k))\oplus \cdots  \oplus  \min(m_{\ell},f_{\ell}(n_1,\ldots  ,n_k))\oplus n_1\oplus \cdots   \oplus n_k.
   \end{align*}
\end{theorem}

\begin{proof}

     We prove this by induction on $\sum_{j=1}^{\ell} m_j+\sum_{i=1}^k n_i$.

    By ${\rm mex}$ rule and induction hypothesis, we have
    \begin{align*}
        &g(\MCHN(f_1,\ldots  ,f_{\ell},n_1,\ldots  ,n_k,m_1,\ldots  ,m_{\ell})) \\
        =& {\rm mex} (\{g(\MCHN(f_1,\ldots  ,f_{\ell},n_1,\ldots  ,n_i',\ldots  ,n_k,m_1,\ldots  ,m_{\ell})), \\
        &g(\MCHN(f_1,\ldots  ,f_{\ell},n_1,\ldots  ,n_k,m_1,\ldots  ,m_j',\ldots  ,m_{\ell})) \\
        &\mid 1\le i \le k, 0\le n_i'< n_i, 1 \le j \le {\ell}, 0 \le m_j'<\min(m_j,f_j(n_1,\ldots  ,n_k) )\}) \\
         =&{\rm mex} (\{\min(m_1,f_1(n_1,\ldots  ,n_i',\ldots  ,n_k))\oplus \cdots   \oplus \min(m_{\ell},f_{\ell}(n_1,\ldots  ,n_i',\ldots  ,n_k))\oplus n_1\oplus \cdots   \oplus n'_i \oplus \cdots   \oplus n_k, \\
         &\min(m_1,f_1(n_1,\ldots  ,n_i,\ldots  ,n_k))\oplus \cdots   \oplus m_j' \oplus \cdots   \oplus \min(m_{\ell},f_{\ell}(n_1,\ldots  ,n_i,\ldots  ,n_k)) \oplus n_1\oplus \cdots   \oplus n_k \\
         &\mid 1\le i \le k, 0\le n_i'< n_i,  1 \le j \le {\ell}, 0 \le m_j'<\min(m_j,f_j(n_1,\ldots  ,n_k)) \}).
    \end{align*}

    By Lemma~\ref{lemma:equality_f(x)f(y)multif}, 
    \begin{align*}
        &\{\min(m_1,f_1(n_1,\ldots  ,n_i',\ldots  ,n_k))\oplus \cdots   \oplus \min(m_{\ell},f_{\ell}(n_1,\ldots  ,n_i',\ldots  ,n_k))\oplus n_1\oplus \cdots   \oplus n'_i \oplus \cdots   \oplus n_k, \\
        &\min(m_1,f_1(n_1,\ldots  ,n_i,\ldots  ,n_k))\oplus \cdots   \oplus m_j' \oplus \cdots   \oplus \min(m_{\ell},f_{\ell}(n_1,\ldots  ,n_i,\ldots  ,n_k)) \oplus n_1\oplus \cdots   \oplus n_k \\
        &\mid 1\le i \le k, 0\le n_i'< n_i, 1 \le j \le {\ell}, 0 \le m_j'<\min(m_j,f_j(n_1,\ldots  ,n_k)) \} \\
        =&\{\min(m_1,f_1(n_1,\ldots  ,n_i,\ldots  ,n_k))\oplus \cdots   \oplus \min(m_{\ell},f_{\ell}(n_1,\ldots  ,n_i,\ldots  ,n_k)) \oplus n_1\oplus \cdots   \oplus n'_i \oplus \cdots   \oplus n_k, \\
        &\min(m_1,f_1(n_1,\ldots  ,n_i,\ldots  ,n_k))\oplus \cdots   \oplus m_j' \oplus \cdots   \oplus \min(m_{\ell},f_{\ell}(n_1,\ldots  ,n_i,\ldots  ,n_k)) \oplus n_1\oplus \ldots   \oplus n_k \\
        &\mid 1\le i \le k, 0\le n_i'< n_i, 1 \le j \le {\ell}, 0 \le m_j'<\min(m_j,f_j(n_1,\ldots  ,n_k)) \}, 
    \end{align*}
    and then 
    \begin{align*}
        &g(\MCHN(f_1,\ldots  ,f_{\ell},n_1,\ldots  ,n_k,m_1,\ldots  ,m_{\ell})) \\
    =&{\rm mex} (\{\min(m_1,f_1(n_1,\ldots  ,n_k))\oplus \cdots   \oplus \min(m_{\ell},f_{\ell}(n_1,\ldots  ,n_k)) \oplus n_1\oplus \cdots   \oplus n'_i \oplus \cdots   \oplus n_k, \\
    &\min(m_1,f_1(n_1,\ldots  ,n_i,\ldots  ,n_k))\oplus \cdots   \oplus m_j' \oplus \cdots   \oplus \min(m_{\ell},f_{\ell}(n_1,\ldots  ,n_i,\ldots  ,n_k)) \oplus n_1\oplus \cdots   \oplus n_k \\
    &\mid 1\le i \le k, 0\le n_i'< n_i, 1 \le j \le {\ell}, 0 \le m_j'<\min(m_j,f_j(n_1,\ldots  ,n_k))  \}) \\
    =&g(\MRCG(\min(m_1,f_1(n_1,\ldots  ,n_k)),\ldots  ,\min(m_{\ell},f_{\ell}(n_1,\ldots  ,n_k)),n_1,\ldots  ,n_k)) \\
    =&\min(m_1,f_1(n_1,\ldots  ,n_k))\oplus \cdots   \oplus \min(m_{\ell},f_{\ell}(n_1,\ldots  ,n_k))\oplus n_1\oplus \cdots   \oplus n_k.
    \end{align*}

\end{proof}

\subsection{Necessary Condition in higher dimension}

In order to prove the necessary condition in the case where we have multiple functional dimensions, it is enough to prove it in the case where we have only one functional dimension.

In this subsection, we prove that the tasty condition is a necessary condition to have $$g(\CHN(f,n_1,\ldots  ,n_k,m))=\min(m,f(n_1,\ldots  ,n_k))\oplus n_1\oplus \cdots   \oplus n_k.$$

In this subsection, we take a function $f$, monotonous according to all dimensions such that for all $n_1,\ldots  ,n_k$,$m$ in $\mathbb{Z}_{\ge 0}$,  $$g(\CHN(f,n_1,\ldots  ,n_k,m))=\min(m,f(n_1,\ldots  ,n_k))\oplus n_1 \oplus \cdots   \oplus n_k.$$






\begin{lemma}
\label{lemma:innegualitymultid}
    For any $i,x_1,\ldots  ,x_i,x_i',\ldots  ,x_k,y \in \mathbb{Z}_{\ge 0}$,if $x_i\ne x_i'$ then 
    $$\min(y,f(x_1,\ldots  ,x_i,\ldots  ,x_k))\oplus x_i \ne \min(y,f(x_1,\ldots  ,x_i',\ldots  ,x_k))\oplus x_i'.$$
\end{lemma}

\begin{proof}
    Let $i,x_1,\ldots  ,x_i,x_i',\ldots  ,x_k,y\in \mathbb{Z}_{\ge 0}$ with $x_i\ne x_i'$. Suppose without loss of generality that $x_i>x_i'$.
    As for all $n_1,\ldots  ,n_k, m \in \mathbb{Z}_{\ge 0}$,  
    $$g(\CHN(f,n_1,\ldots  ,n_k,m))=\min(m,f(n_1,\ldots  ,n_k))\oplus n_1 \oplus \cdots   \oplus n_k,$$
    we have 
    $$g(\CHN(f,x_1,\ldots  ,x_i,\ldots  ,x_k,y))=\min(y,f(x_1,\ldots  ,x_i,\ldots  ,x_k))\oplus x_1 \oplus \cdots  \oplus x_i \oplus \cdots   \oplus x_k$$
    and 
    $$g(\CHN(f,x_1,\ldots  ,x_i',\ldots  ,x_k,y))=\min(y,f(x_1,\ldots  ,x_i',\ldots  ,x_k))\oplus x_1 \oplus \cdots  \oplus x_i' \oplus \cdots  \oplus x_k.$$ 
    However, there is a move from $\CHN(f,x_1,\ldots  ,x_i,\ldots  ,x_k,y)$ to $\CHN(f,x_1,\ldots  ,x_i',\ldots  ,x_k,y)$ then the Sprague--Grundy values must be different and $$\min(y,f(x_1,\ldots  ,x_i,\ldots  ,x_k))\oplus x_1 \oplus \cdots  \oplus x_i \oplus \cdots   \oplus x_k \ne \min(y,f(x_1,\ldots  ,x_i',\ldots  ,x_k))\oplus x_1 \oplus \cdots  \oplus x_i' \oplus \cdots  \oplus x_k.$$
    Since $f:x \xrightarrow{} x\oplus z$ is a bijection for all $z$, $$\min(y,f(x_1,\ldots  ,x_i,\ldots  ,x_k))\oplus x_i \ne \min(y,f(x_1,\ldots  ,x_i',\ldots  ,x_k))\oplus x_i'.$$
\end{proof}

\begin{lemma}
 \label{lemma:divide_by_twomultid}
    If for $x_1,\ldots  ,x_k,m,j,i \in \mathbb{Z}_{\ge 0}$,there exists $y$ such as for all $x_i2^j\le z<(x_i+1)2^j$ we have $$y2^j\le \min(f(x_1,\ldots  ,x_{i-1},z,x_{i+1}\ldots  ,x_k),m)<(y+1)2^j,$$ then 
    \begin{itemize}
        \item either for all $x_i2^j\le z'<(x_i+1)2^j$, $$(2y)2^{j-1}\le \min(f(x_1,\ldots  ,x_{i_1},z',x_{i+1}\ldots  ,x_k),m)<(2y+1)2^{j-1},$$
        \item or for all $x_i2^j\le z'<(x_i+1)2^j$, $$(2y+1)2^{j-1}\le \min(f(x_1,\ldots  ,x_{i-1},z',x_{i+1}\ldots  ,x_k),m)<(2y+2)2^{j-1}.$$
    \end{itemize}
\end{lemma}

\begin{proof}
    We take $x_1,\ldots  ,x_k,m,j,i,y \in \mathbb{Z}_{\ge 0}$ such as for all $x_i2^j\le z<(x_i+1)2^j$ we have $$y2^j\le \min(f(x_1,\ldots  ,x_{i-1},z,x_{i+1}\ldots  ,x_k),m)<(y+1)2^j.$$

    If $x_i2^j\le z < (2x_i+1)2^{j-1}$ and $$(2y+1)2^{j-1} \le \min(f(x_1,\ldots  ,x_{i-1},z,x_{i+1}\ldots  ,x_k),m) < (y+1)2^{j},$$ then $$((2x_i\oplus 2y)+1)2^{j-1}\le z\oplus \min(f(x_1,\ldots  ,x_{i-1},z,x_{i+1}\ldots  ,x_k),m)<((2x_i\oplus 2y)+2)2^{j-1}.$$ 
    
    If $(2x_i+1)2^{j-1}\le z < (x_i+1)2^{j}$ and $$y2^{j} \le \min(f(x_1,\ldots  ,x_{i-1},z,x_{i+1}\ldots  ,x_k),m) < (2y+1)2^{j-1},$$ then $$((2x_i\oplus 2y)+1)2^{j-1}\le z\oplus \min(f(x_1,\ldots  ,x_{i-1},z,x_{i+1}\ldots  ,x_k),m)<((2x_i\oplus 2y)+2)2^{j-1}.$$


Without loss of generality, we assume that $f$ is decreasing according to the $i$-th coordinate. For the case $f$ is increasing according to the $i$-th coordinate, we can prove this by using a similar argument. \\

     Suppose that $$\min(f(x_1,\ldots  ,x_{i-1},(2x_i+1)2^{j-1},x_{i+1}\ldots  ,x_k),m)<(2y+1)2^{j-1},$$ then as $f$ is decreasing according to $x_i$ for all $(2x_i+1)2^{j-1}\le z<(x_i+1)2^{j}$, $$y2^{j} \le \min(f(x_1,\ldots  ,x_{i-1},z,x_{i+1}\ldots  ,x_k),m) < (2y+1)2^{j-1}$$ and then $$((2x_i\oplus 2y)+1)2^{j-1}\le z\oplus \min(f(x_1,\ldots  ,x_{i-1},z,x_{i+1}\ldots  ,x_k),m)<((2x_i\oplus 2y)+2)2^{j-1}.$$
     
     However, there are $2^{j-1}$ different values between $((2x_i\oplus 2y)+1)2^{j-1}$ and $((2x_i\oplus 2y)+2)2^{j-1}-1$ and there are also exactly $2^{j-1}$ different possible values for $z$.
     
     As for all $z,z'$, $$z\oplus \min(f(x_1,\ldots  ,x_{i-1},z,x_{i+1}\ldots  ,x_k),m) \ne z'\oplus \min(f(x_1,\ldots  ,x_{i-1},z',x_{i+1}\ldots  ,x_k),m)$$ by Lemma~\ref{lemma:innegualitymultid}, there is no other $z$ such as $$((2x_i\oplus 2y)+1)2^{j-1}\le z\oplus \min(f(x_1,\ldots  ,x_{i-1},z,x_{i+1}\ldots  ,x_k),m)<((2x_i\oplus 2y)+2)2^{j-1}.$$

     Then there is no $z$ such as $x_i2^j\le z < (2x_i+1)2^{j-1}$ and $$(2y+1)2^{j-1} \le \min(f(x_1,\ldots  ,x_{i-1},z,x_{i+1}\ldots  ,x_k),m) < (y+1)2^{j},$$ which means that $f(x_1,\ldots,x_i2^{j},\ldots,x_k)< (2y+1)2^{j-1}$. Then as $f$ is decreasing according to the $i$-th coordinate, for all $x_i2^j\le z'<(x_i+1)2^j$, $$(2y)2^{j-1}\le \min(f(x_1,\ldots  ,x_{i-1},z',x_{i+1}\ldots  ,x_k),m)<(2y+1)2^{j-1}.$$\\

    Suppose that $$\min(f(x_1,\ldots  ,x_{i-1},(2x_i+1)2^{j-1},x_{i+1}\ldots  ,x_k),m)\ge(2y+1)2^{j-1},$$ then as $f$ is decreasing for all $x_i2^j\le z<(2x_i+1)2^{j-1}$, $$(2y+1)2^{j-1} \le \min(f(x_1,\ldots  ,x_{i-1},z,x_{i+1}\ldots  ,x_k),m) < (y+1)2^{j}$$ and then $$((2x_i\oplus 2y)+1)2^{j-1}\le z\oplus \min(f(x_1,\ldots  ,x_{i-1},z,x_{i+1}\ldots  ,x_k),m)<((2x_i\oplus 2y)+2)2^{j-1}.$$
    
    However, there are $2^{j-1}$ different values between $((2x_i\oplus 2y)+1)2^{j-1}$ and $((2x_i\oplus 2y)+2)2^{j-1}-1$ and there are also exactly $2^{j-1}$ different possible values for $z$. 
    
    As for all $z,z'$, $$z\oplus \min(f(x_1,\ldots  ,x_{i-1},z,x_{i+1}\ldots  ,x_k),m) \ne z'\oplus \min(f(x_1,\ldots  ,x_{i-1},z',x_{i+1}\ldots  ,x_k),m)$$ by Lemma~\ref{lemma:innegualitymultid}, there are no other $z$  as $$((2x_i\oplus 2y)+1)2^{j-1}\le z\oplus \min(f(x_1,\ldots  ,x_{i-1},z,x_{i+1}\ldots  ,x_k),m)<((2x_i\oplus 2y)+2)2^{j-1}.$$ 
    
    Then there are no $z$  as $(2x_i+1)2^{j-1}\le z < (x_i+1)2^{j}$ and $$y2^{j} \le \min(f(x_1,\ldots  ,x_{i-1},z,x_{i+1}\ldots  ,x_k),m) < (2y+1)2^{j-1},$$ which means that $f(x_1,\ldots,x_{i-1},(x_i+1)2^{j}-1,x_{i+1},\ldots,x_k)\ge (2y+1)2^{j-1}$. Then as $f$ is decreasing according to the $i$-th coordinate, for all $x_i2^j\le z'<(x_i+1)2^j$, $$(2y+1)2^{j-1}\le \min(f(x_1,\ldots  ,x_{i-1},z',x_{i+1}\ldots  ,x_k),m)<(2y+2)2^{j-1}.$$

\end{proof}

\begin{lemma}
\label{lemma:bornonf(z)multid}
    For any $x_1,\ldots  x_k,i,j,m \in \mathbb{Z}_{\ge 0},$ there exists $y \in \mathbb{Z}_{\ge 0}$ such as for all $x_i2^j\le z<(x_i+1)2^j$, we have $$y2^{j-1}\le \min(f(x_1,\ldots  ,x_{i-1},z,x_{i+1}\ldots  ,x_k),m)<(y+1)2^{j-1}.$$
\end{lemma}

\begin{proof}
    
We take $x_1,\ldots  ,x_k,i,j,m \in \mathbb{Z}_{\ge 0}$.
We prove this by decreasing induction on $j$.

If $j>\log_2(m)+1$, then for all $z$, $$\min(f(x_1,\ldots  ,x_{i-1},z,x_{i+1}\ldots  ,x_k),m)\le m < 2^{j-1}.$$ 

If we take $y=0$, we have for all $z$, $$y2^{j-1}=0\le \min(f(x_1,\ldots  ,x_{i-1},z,x_{i+1}\ldots  ,x_k),m)<1\times2^{j-1}.$$

If $j\le \log_2(m)+1$ and the lemma works for $j+1$,  then there exists $y' \in \mathbb{Z}_{\ge 0}$ such that for all $\left\lfloor \frac{x_i}{2} \right\rfloor 2^{j+1}\le z<(\left\lfloor \frac{x_i}{2} \right\rfloor+1)2^{j+1},$ we have $$y'2^{j}\le \min(f(x_1,\ldots  ,x_{i-1},z,x_{i+1}\ldots  ,x_k),m)<(y'+1)2^{j}.$$  Then by Lemma~\ref{lemma:divide_by_twomultid}, either for all $x_i2^{j}\le z<(x_i+1)2^{j}$, $$(2y')2^{j-1}\le \min(f(x_1,\ldots  ,x_{i-1},z,x_{i+1}\ldots  ,x_k),m)<(2y'+1)2^{j-1}$$ or for all $x_i2^{j}\le z<(x_i+1)2^{j}$, $$(2y'+1)2^{j-1}\le \min(f(x_1,\ldots  ,x_{i-1},z,x_{i+1}\ldots  ,x_k),m)<(2y'+2)2^{j-1}$$ giving in both cases what we want.

By decreasing induction, the lemma holds. 
\end{proof}

\begin{theorem}
\label{th:necessarymultid}
    $f$ satisfies the tasty condition.
\end{theorem}

\begin{proof}


     We take $z,z' \in \mathbb{Z}_{\ge 0} $ and a positive integer $j$ such that $\left\lfloor \frac{z}{2^j} \right\rfloor=\left\lfloor \frac{z'}{2^j} \right\rfloor$.

     We take $x_1,\ldots  ,x_k,i \in \mathbb{Z}_{\ge 0}$ and also take $$m>\min(f(x_1,\ldots  ,x_{i-1},z,x_{i+1}\ldots  ,x_k),f(x_1,\ldots  ,x_{i-1},z',x_{i+1}\ldots  ,x_k)).$$

    By Lemma~\ref{lemma:bornonf(z)multid}, there exists $y$ such as for all $\left\lfloor \frac{z}{2^j} \right\rfloor2^j\le z''<(\left\lfloor \frac{z}{2^j} \right\rfloor+1)2^j$, $y2^{j-1}\le f(z'')<(y+1)2^{j-1}$.
    
    We have $\left\lfloor \frac{z}{2^j} \right\rfloor2^j\le z,z'<(\left\lfloor \frac{z}{2^j} \right\rfloor+1)2^j$ as $\left\lfloor \frac{z}{2^j} \right\rfloor=\left\lfloor \frac{z'}{2^j} \right\rfloor$. 
    Then $$y2^{j-1}\le \min(f(x_1,\ldots  ,x_{i-1},z,x_{i+1}\ldots  ,x_k),m),\min(f(x_1,\ldots  ,x_{i-1},z',x_{i+1}\ldots  ,x_k),m)<(y+1)2^{j-1}$$ and as $$m>\min(f(x_1,\ldots  ,x_{i-1},z,x_{i+1}\ldots  ,x_k),f(x_1,\ldots  ,x_{i-1},z',x_{i+1},\ldots  ,x_k)),$$ $$ \left\lfloor \frac{f(x_1,\ldots  ,x_{i-1},z,x_{i+1},\ldots  ,x_k)}{2^{j-1}} \right\rfloor=\left\lfloor \frac{f(x_1,\ldots  ,x_{i-1},z',x_{i+1},\ldots  ,x_k)}{2^{j-1}} \right\rfloor=y.$$ 
    
\end{proof}

\subsection{Necessary Condition with many functional dimension}

In this subsection, we prove that the tasty condition is a necessary condition to have 
\begin{align*}
    &g(\MCHN(f_1,\ldots  ,f_{\ell},n_1,\ldots  ,n_k,m_1,\ldots  ,m_{\ell})) \\
    =&\min(m_1,f_1(n_1,\ldots  ,n_k))\oplus \cdots   \oplus \min(m_{\ell},f_{\ell}(n_1,\ldots  ,n_k))\oplus n_1\oplus \cdots   \oplus n_k.
\end{align*}

For all this subsection, we take functions $f_j$ monotonous in all dimensions such that for all $n_1,\ldots  ,n_k$,$m_1,\ldots  ,m_{\ell}$ in $\mathbb{Z}_{\ge 0}$,  
\begin{align*}
    &g(\MCHN(f,n_1,\ldots  ,n_k,m_1,\ldots  ,m_{\ell})) \\
    =&\min(m_1,f_1(n_1,\ldots  ,n_k))\oplus \cdots   \oplus \min(m_{\ell},f_{\ell}(n_1,\ldots  ,n_k))\oplus n_1\oplus \cdots   \oplus n_k .
\end{align*}

\begin{theorem}
\label{th:necessarymultif}
    For any $j\le {\ell},$ $f_j$ satisfies the tasty condition.
\end{theorem}

\begin{proof}
    If we take $m_{j'}=0$ for $j' \ne j$ then by Theorem~\ref{th:necessarymultid}, $f_j$ satisfies the tasty condition.
    
\end{proof}

\begin{theorem}
\label{th:necessary_and_sufficientMCHN}
For all $x_1,\ldots  ,x_k,y_1,\ldots  ,y_{\ell}$ in $\mathbb{Z}_{\ge 0}$, $$g(\MCHN(f_1,\ldots  ,f_{\ell}, x_1,\ldots  ,x_k,y_1,\ldots  ,y_{\ell}))=g(\MRCG(x_1,\ldots  ,x_k,y_1,\ldots  ,y_{\ell}))$$ if $f_1,\ldots,f_\ell$ satisfy the tasty condition.

    If all $f_j$ are monotonous according to all dimensions, $$g(\MCHN(f_1,\ldots  ,f_{\ell}, x_1,\ldots  ,x_k,y_1,\ldots  ,y_{\ell}))=g(\MRCG(x_1,\ldots  ,x_k,y_1,\ldots  ,y_{\ell}))$$ for all $x_1,\ldots  ,x_k,y_1,\ldots  ,y_{\ell}$ in $\mathbb{Z}_{\ge 0}$ if and only if $f_1,\ldots,f_\ell$ satisfy the tasty condition.
\end{theorem}

\begin{proof}
    By Theorem~\ref{th:necessarymultif}, if $f_1,\ldots,f_\ell$ are monotonous according to all dimensions, the tasty condition is a necessary condition to have for all $x_1,\ldots  ,x_k,y_1,\ldots  ,y_{\ell}\in \mathbb{Z}_{\ge 0}$, $$g(\MCHN(f_1,\ldots  ,f_{\ell}, x_1,\ldots  ,x_k,y_1,\ldots  ,y_{\ell}))=g(\MRCG(x_1,\ldots  ,x_k,y_1,\ldots  ,y_{\ell})).$$

    By Theorem~\ref{th:sufficientmultif}, the tasty condition is a sufficient condition to have for all $x_1,\ldots  ,x_k,y_1,\ldots  ,y_{\ell}\in \mathbb{Z}_{\ge 0}$, $$g(\MCHN(f_1,\ldots  ,f_{\ell}, x_1,\ldots  ,x_k,y_1,\ldots  ,y_{\ell}))=g(\MRCG(x_1,\ldots  ,x_k,y_1,\ldots  ,y_{\ell})).$$ 
\end{proof}


Finally, we extend Theorem \ref{thm:fitting} to  multidimensional cases.

\begin{theorem}
    For all monotonous functions $g,f_2,\ldots f_\ell$  according to all the dimensions, let $$G_{x_1,\ldots,x_k}(z_1,\ldots, z_k)=2^{\left\lceil \log_2 \left(\max_{x_i'\le x_i}(g(x'_1,\ldots ,x'_k)) \right)\right\rceil}-g(z_1,\ldots, z_k)-1.$$ If $\MCHN(g,f_2,\ldots f_\ell,x_1,\ldots,x_k,y_1,\ldots,y_{\ell})$ is tasty for all nonnegative integers $x_1,\ldots,x_k,y_1,\ldots,y_{\ell}$, then $\MCHN(G_{x_1,\ldots,x_k},f_2,\ldots f_\ell,x_1,\ldots,x_k,y_1,\ldots,y_{\ell})$ is tasty for all nonnegative integers  $x_1,\ldots,x_k,y_1,\ldots,y_{\ell}$.
\end{theorem}

\begin{proof}
    If for all $x_1,\ldots,x_k,y_1,\ldots,y_{\ell}$ in $\mathbb{Z}_{\ge 0}$, $\MCHN(g,f_2,\ldots f_\ell,x_1,\ldots,x_k,y_1,\ldots,y_{\ell})$ is tasty then by Theorem~\ref{th:necessary_and_sufficientMCHN}, $g,f_1,\ldots,f_\ell$ satisfy the tasty condition.

    Let $g'_{x_1,\ldots,x_k}(z_1,\ldots,z_k)=\min(g(z_1,\ldots, z_k),\max_{x_i'\le x_i}(g(x'_1,\ldots ,x'_k)))$ and $$G'_{x_1,\ldots,x_k}(z_1,\ldots,z_k)=2^{\left\lceil \log_2 \left(\max_{x_i'\le x_i}(g'_{x_1,\ldots,x_k}(x'_1,\ldots ,x'_k)) \right)\right\rceil}-g'_{x_1,\ldots,x_k}(z_1,\ldots, z_k)-1$$ then $g'$ is bounded and by Theorem~\ref{theorem:mintasty} satisfies the tasty condition and by Theorem~\ref{theorem:bounded_fitting}, $G'_x$ satisfies the tasty condition too.
    
    By Theorem~\ref{th:necessary_and_sufficientMCHN}, $\MCHN(G'_{x_1,\ldots,x_k},f_2,\ldots f_\ell,x_1,\ldots,x_k,y_1,\ldots,y_{\ell})$  is tasty but for $z_1\le x_1,\ldots,z_k \le x_k$, $$G'_{x_1,\ldots,x_k}(z_1,\ldots, z_k)=G_{x_1,\ldots,x_k}(z_1,\ldots, z_k)$$ and then $\MCHN(G_{x_1,\ldots,x_k},f_2,\ldots f_\ell,x_1,\ldots,x_k,y_1,\ldots,y_{\ell})$ is tasty.
\end{proof}

\section*{Acknowledgments}

The authors are deeply grateful to Professor Takeaki Uno for providing this valuable research opportunity and an excellent environment in which to conduct the research.
This work is partially
supported by The INOUE ENRYO Memorial Grant, TOYO
University.


\end{document}